%% file: report.tex
\documentclass[english,plain]{article}

\RequirePackage{xcolor,graphicx}

\RequirePackage{amsmath,amsfonts,amssymb,amsthm,mathtools}
\RequirePackage{bm}

\RequirePackage[a4paper]{geometry}

\DeclareMathOperator*{\tr}{\operatorname{tr}}
\DeclareMathOperator*{\spanop}{\operatorname{span}}
\DeclareMathOperator*{\dof}{\operatorname{dof}}
\newcommand{\defeq}{\vcentcolon=}
\newcommand{\norm}[1]{{\left\Vert {#1} \right\Vert}}
\newcommand{\abs}[1]{{\left\vert {#1} \right\vert}}

\newtheorem{theorem}{Theorem}

\newtheorem{problem}{Problem}
\newtheorem{remark}{Remark}

\usepackage[normalem]{ulem}

\graphicspath{{./}}

\title{A multi-layer reduced model for flow in porous media with a fault and
surrounding damage zones}

\author{Alessio Fumagalli \and Anna Scotti}

\begin{document}

\maketitle

\begin{abstract}
    In this work we present a new conceptual model to describe  fluid
    flow in a porous media system in presence of a large
    fault. {Geological faults are often
    modeled simply as interfaces in the rock matrix, but they are complex
    structure where the high strain core is surrounded by the so called damage
    zones, characterized by the presence of smaller fractures which enhance the
    permeability of the medium. To obtain reliable simulation outcomes these
    damage zone, as well as the fault, have to be accurately described}. The
    new model {proposed in this work} considers {both} these two
    {regions} as lower dimensional and embedded in the rock matrix. The model is
    presented, analyzed, and tested in several configurations to prove its
    robustness and ability to capture many important features, such as hight
    contrast and heterogeneity of permeability.
\end{abstract}

\section{Introduction}

The accurate description and simulation of fluid flow in {geological} porous media are of
paramount importance for several industrial applications, such as: CO$_2$
injection and sequestration, geothermal exploitation, oil migration and
recovery, and prevention of groundwater contamination due to nuclear waste
disposal, just to name a few. See \cite{bear1993flow,Karimi-Fard2006,Hui2008}. Despite {the relevance of the topic}, a number of challenges are
not yet fully solved, in particular related to the presence of faults and
fractures in the domain of interest. Faults and fractures are the portion of the
porous media where the rock has been broken due to geological movements of the
upper crust. In this study we consider only large faults.

{A tectonic fault is the result of a relative displacement of two parts of the
upper crust happened over geological eras. This displacement is accommodate by a high strain region, the fault core,
surrounded on both sides by a highly fractured region, the damage zone.}
These layer contain several fractures, on a much smaller scale than
the fault, which may alter the local properties of the flow path. Faults have a thickness which is several order of magnitude
smaller than any other characteristic sizes in the porous domain, however their
physical properties may greatly differ from the porous media.
Due to infilling processes and chemical reactions these objects
may be partially or completely occluded, thus
they can behave as preferential conduits or geological barriers for the fluid
flow. The surrounding damage zone may or may not had experienced the same
processes behaving similarly to, or differently from than the related fault.
The development of accurate conceptual models is a key factor to be able to include
these objects and their effect in a simulation code and obtain reliable outcomes
and predictions. {The aim of this work is to devise a new effective conceptual model to account for multiple thin regions in porous media.}

One possibility to account for the presence of faults is
{to characterize this region with a different permeability}, but its
{small} thickness makes difficult or even unrealistic its
inclusion in the grid representing the rock matrix.
A common approach, introduced in \cite{Jaffre2002,Tunc2012,Faille2014a,Faille2014}, is to consider faults as a lower
dimensional objects and derive a new conceptual model to describe the flow and
pressure behavior inside and across these objects. This approach has
been successfully applied to different kind of physics, ranging from advection
of a passive scalar, heat transport, multi-phase flow. An incomplete list of references is the following
\cite{Martin2005,Reichenberger2006,Hoteit2008,Angot2009,Morales2010,Jaffre2011,DAngelo2011,Frih2011,Formaggia2012,Morales2012,Fumagalli2012d,Sandve2012,Knabner2014,Schwenck2015,DelPra2015a,Ahmed2015,Brenner2015,Karimi-Fard2016,Brenner2016a,Antonietti,Scotti2017,Flemisch2016a,Boon2018,Brenner2018,Nordbotten2018}.

In this work we extend the previously introduced models to consider also the
damage zone as a lower dimensional objects which are connected on one side to
the rock matrix and, on the other side, to the fault. {The aim is to be able to}
simulate different scenarios where the rock matrix, damage zone, and fault may
have different permeability values {without resorting to extreme mesh refinement
to capture the thickness of the damage and core zones. Moreover this multilayer
approach allows for different apertures and asymmetries across the fault, unlike
the previous models}.

The numerical discretization is based on the classical
Raviart-Thomas-N\'ed\'elec approximation for the flux field and a constant
piecewise representation of the pressure. The resulting scheme is locally mass
conservative and is thus suitable {for} coupling with a transport problem.
For clarity in the exposition, we consider only one single fault, {since} the case of
intersecting faults requires additional model complexities.

For the implementation of the numerical examples, we have used the library
PorePy \cite{Keilegavlen2017a}, which is a simulation tool written in Python for fractured
and deformable porous media. The numerical tests presented in this paper are available in the GitHub repository of the
library. The main contribution of this work is the
introduction of the multi-layers interface law, valid for any dimension. Even if
we present an approximation and analysis based on the lowest order Raviart-Thomas-N\'ed\'elec, the implementation is
agnostic with respect to the numerical scheme. It is {indeed} possible to use {any other scheme}
present in the library, like two and multi point flux
approximation or the mixed virtual element method.

The paper is organised as follow: in Section \ref{sec:mathematical_model} both
the equi and mixed-dimensional mathematical models are introduced and discussed.
Section \ref{sec:weak_formulation} deals with the weak formulation of the
mixed-dimensional problem: functional spaces, weak problem, and it well
posedness. In Section \ref{sec:numerical_discretization} we briefly describe the
numerical scheme and how to handle the mixed-dimensional nature of the problem.
Section \ref{sec:numerical_examples} contains two numerical test to validate the
proposed model. Finally, Section
\ref{sec:conclusion} is devoted to conclusions.

\section{Mathematical problem}\label{sec:mathematical_model}

In this section we present the mathematical models considered to describe the
pressure and Darcy velocity governed by the single-phase {Darcy} flow equations.
In Subsection \ref{subsec:equidim}, we recall the classical formulation where
the fault and the damage zone are considered equi-dimensional with respect to
the rock matrix, {i.e. are represented as two dimensional in a two dimensional
porous matrix, and three dimensional in 3D.} Their thickness is thus explicitly
represented in the domain and {captured by the} computational grid. By
considering a model reduction approach, Subsection \ref{subsec:lowerdim}
presents the new formulation where the fault and damage zone are now
represented as lower-dimensional objects and new equations are introduced.

\subsection{Equi-dimensional model}\label{subsec:equidim}

{Let $\Omega$ denote the rock host matrix, $\mu$ denote the damage zone and $\gamma$ the fault core.}
The interfaces $M$ and $\Gamma$ between these objects can
be {defined} as $\overline{M } = \overline{\partial \Omega} \cap \overline{\mu}$ and
$\overline{\Gamma}
=
\overline{\mu} \cap \overline{\gamma}$, respectively. We define $\partial \Omega_{\rm int}$ as the part of the boundary of $\Omega$ which is in contact with $M$, and
with $\partial
\Omega$  the part of the boundary of
$\Omega$ which is not
in contact with $M$. We further subdivide
$\partial \Omega$ into two parts, $\partial \Omega_{p}$
and $\partial \Omega_{u}$, such that
$\overline{\partial \Omega
} = \overline{\partial \Omega_{p}} \cup \overline{\partial \Omega_{u}}$ with
$\partial \Omega_{p} \neq \emptyset$ and
$\emptyset = \mathring{\partial \Omega_{p}} \cup \mathring{\partial
\Omega_{u}}$. In $\partial \Omega_{p}$ we will impose natural boundary
conditions, while in $\partial \Omega_{u}$ essential boundary conditions.
In the same way, we define the boundary of the damage zone $\mu$ as $\partial
\mu_{\rm int}$ and $\partial
\mu$, being the internal and external portion of the boundary of $\mu$,
respectively. The external part $\partial \mu$ can be divided into two parts
$\partial \mu_{p}$
and $\partial \mu_{u}$ such that
$\overline{\partial \mu} = \overline{\partial \mu_{p}} \cup \overline{\partial
\mu_{u}}$ with
$\partial \mu_{p} \neq \emptyset$ and
$\emptyset = \mathring{\partial \mu_{p}} \cup \mathring{\partial
\mu_{u}}$. In $\partial \mu_{p}$ we will impose natural boundary
conditions, while in $\partial \mu_{u}$ essential boundary conditions.
The same nomenclature is introduced for the fault $\gamma$.
We define the unit normal $\bm{n}$, associated
with $\partial \Omega_{\rm int}$, which points from $\Omega$
to $M$. Similarly, we introduce the unit normal $\bm{n}_\mu$, associated with
$\partial \mu_{\rm int}$, which points from $\mu$ to $\Gamma$.
Finally, $\bm{\nu}$ is the external outward unit normal of the domain. See Figure \ref{fig:domain_equi} as an example.

We assume that the rock matrix $\Omega$, damage zone $\mu$, and fault $\gamma$
have the same spatial dimension. The damage zone and fault are characterized
by one of the dimension to be much (orders of magnitude) smaller than the
others.
\begin{figure}[htb]
    \centering
    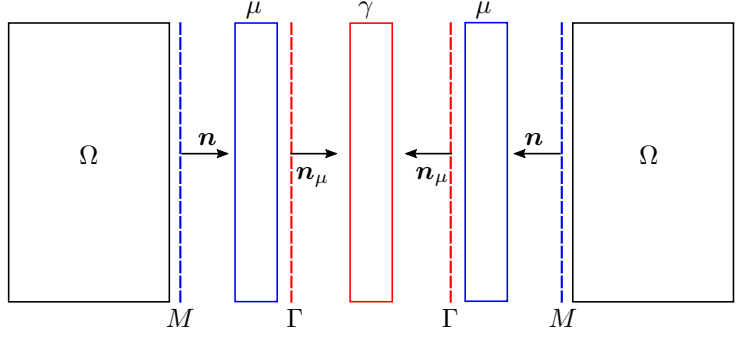
    \caption{Equi-dimensional representation of the rock matrix $\Omega$,
    damage zone $\mu$, and fault $\gamma$.}%
    \label{fig:domain_equi}
\end{figure}
We are interested in the mathematical description of the Darcy velocity and
pore pressure described by the Darcy problem. We indicate with $p_\Upsilon$ the
pressure and with $\bm{u}_\Upsilon$
the Darcy velocity, on the portion (rock matrix, damage zone, or fault)
$\Upsilon$ of
the problem, being $\Upsilon$ equal to $\Omega$, $\mu$, and $\gamma$.  The
system of equations reads: find $(\bm{u}_\Upsilon, p_\Upsilon)$ such that
\begin{subequations}\label{eq:model_equi}
\begin{gather}\label{eq:model_equi_eq}
    \begin{aligned}
        &
        \begin{aligned}
            &\sigma_{\Upsilon}^2 \bm{u}_\Upsilon + \nabla p_\Upsilon = \bm{0}\\
            &\nabla \cdot \bm{u}_\Upsilon + q_\Upsilon = 0
        \end{aligned}
        &&\text{in } \Upsilon\\
        &\tr p_\Upsilon = \overline{p_\Upsilon} &&\text{on } \partial \Upsilon_{p}\\
        &\tr \bm{u}_\Upsilon \cdot \bm{\nu}
        = \overline{u_\Upsilon} &&\text{on } \partial \Upsilon_{u}
    \end{aligned},
\end{gather}
The parameters are the inverse of the permeability $\sigma^2_\Upsilon$, the source or sink term
$q_\Upsilon$, and the
boundary conditions on the pressure $\overline{p_\Upsilon}$ and velocity
$\overline{u_\Upsilon}$. We make use of $\tr$ to indicate the trace operator.
To couple the problem in the three domains, we consider the following
transmission conditions
\begin{gather}\label{eq:model_equi_cc}
    \begin{aligned}
        &\begin{aligned}
            &\tr \bm{u}_\Omega \cdot \bm{n} = \tr \bm{u}_\mu \cdot \bm{n} \\
            &\tr p_\Omega = \tr p_\mu
        \end{aligned}
        && \text{on } M\\
        &\begin{aligned}
            &\tr \bm{u}_\mu \cdot \bm{n}_\mu = \tr \bm{u}_\gamma \cdot \bm{n}_\mu\\
            &\tr p_\mu = \tr p_\gamma
        \end{aligned}
        && \text{on }\Gamma
    \end{aligned}.
\end{gather}
\end{subequations}
Even if not explicitly used in the previous set of equations, we indicate with
$\epsilon_\mu$ and $\epsilon_\gamma$ the thickness of the damage zone and
fault, respectively.
To summarize the equi-dimensional system of equations, we introduce the following
problem.
\begin{problem}[Equi-dimensional model problem]\label{pb:model_equi}
    Find $(\bm{u}_\Upsilon, p_\Upsilon)$ such that
    \eqref{eq:model_equi} is satisfied, for $\Upsilon$ equal to $\Omega$, $\mu$,
    and $\gamma$.
\end{problem}

\subsection{Mixed-dimensional model}\label{subsec:lowerdim}

The {geometrical reduction of the }model  approximates the thin {regions, in our
case} the damage zone
and the fault, by their center line (a lower dimensional object) and derives new equations and coupling
conditions to describe the Darcy velocity and pressure field in the new setting.
See Figure \ref{fig:domain} as an example.
{In this work we follow the reduction procedure described in the literature in} \cite{Alboin2000a,Faille2002,Martin2005}.
To keep the notation simple, we preserve the same notation for the {rock matrix
$\Omega$, the} damage zone
$\mu$ and fault $\gamma$ {even if the domains are geometrically different from
the equi-dimensional case: in particular $\Omega$ is extended up to the center line of the fault core,
while the fault and the damage zone shrink and become overlapped lower dimensional interfaces.}
\begin{figure}[htb]
    \centering
    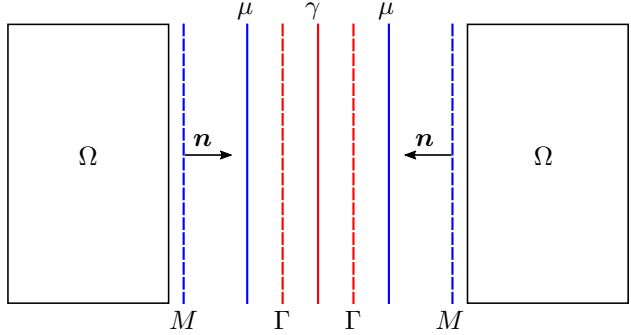
    \caption{Mixed-dimensional representation of the rock matrix $\Omega$,
    damage zone $\mu$,
    and fault $\gamma$.}%
    \label{fig:domain}
\end{figure}

The mixed-dimensional problem
describes the Darcy velocity and pressure field defined on the tangent space of
each domain. At the interfaces, to couple the problems, we consider
traces of the variables and the additional unknown
$u_{\Gamma}$, which can be interpreted as a normal Darcy velocity from the
damage zone
$\mu$ to the fault
$\gamma$. To simplify the notation, we introduce the pressure and Darcy velocity
compounds $p$ and $\bm{u}$, respectively, as
\begin{gather*}
    p = (p_{\Omega}, p_{\mu}, p_\gamma)
    \quad \text{and} \quad
    \bm{u} = (\bm{u}_{\Omega}, \bm{u}_{\mu}, \bm{u}_\gamma,
    u_{\Gamma}).
\end{gather*}
{Note that $\bm{u}_{\mu}$ and  $\bm{u}_{\gamma}$ are fluxes in the lower dimensional layers
and $u_\Gamma$ is a scalar variable representing the normal flux exchanged by lower dimensional objects.}
Following the approaches presented in \cite{Martin2005,Tunc2012,Faille2014a,Faille2014},
we {consider} the Darcy model in each domain separately. {First we consider the rock matrix,}
\begin{subequations}\label{eq:model}
\begin{gather}
    \begin{aligned}
        &
        \begin{aligned}
            &\alpha_{\Omega}^2 \bm{u}_\Omega + \nabla p_\Omega = \bm{0}\\
            &\nabla_\Omega \cdot \bm{u} + q_\Omega = 0
        \end{aligned}
        &&\text{in } \Omega\\
        &\tr p_\Omega = \overline{p_\Omega} &&\text{on } \partial \Omega_{p}\\
        &\tr \bm{u}_\Omega \cdot \bm{\nu}
        = \overline{u_\Omega} &&\text{on } \partial \Omega_{u}
    \end{aligned},
\end{gather}
where $\nabla_{\Omega} \cdot$ is the conservation operator (standard divergence)
in the rock matrix. The parameters are $\alpha_\Omega^2$ which is the inverse of the
rock matrix permeability, $q_\Omega$ which represents a scalar source or sink
term, and
$\overline{p_\Omega}$ and $\overline{u_\Omega}$ denoting the natural (pressure) and
essential (flux) boundary data, respectively.
In the damage zone $\mu$ the problem reads
\begin{gather}\label{2c}
    \begin{aligned}
        &
        \begin{aligned}
            &\alpha_{\mu}^2 \bm{u}_{\mu} + \nabla p_{\mu} = \bm{0}\\
            &\nabla_{\mu} \cdot \bm{u} + q_{\mu} = 0
        \end{aligned}
        && \text{in } \mu\\
        &\tr p_{\mu} = \overline{p_\mu} &&\text{on } \partial \mu_p\\
        &\tr \bm{u}_\mu \cdot \bm{\nu} = \overline{u_\mu} &&\text{on } \partial
        \mu_{u}
    \end{aligned},
\end{gather}
where $\nabla_{\mu} \cdot$ is the conservation operator (mixed-dimensional
divergence) {which, as detailed below, accounts for the tangential divergence in $\mu$ and for
the fluid exchanges between $\mu$ and $\gamma$, and $\mu$ and $\Omega$.}
In \eqref{2c} $\alpha_\mu^2$ is the inverse of the
effective matrix permeability in the damage zone, $q_\mu$ represents a scalar source or sink term,
$\overline{p_\mu}$ and $\overline{u_\mu}$ are the natural and
essential boundary data, respectively.
In the fault core $\gamma$ we have
\begin{gather}\label{2d}
    \begin{aligned}
        &
        \begin{aligned}
            &\alpha_\gamma^2 \bm{u}_\gamma + \nabla p_\gamma = \bm{0}\\
            &\nabla_\gamma \cdot \bm{u} + q_\gamma = 0
        \end{aligned}
        &&\text{in } \gamma\\
        &\tr p_\gamma = \overline{p_\gamma} &&\text{on } \partial \gamma_p\\
        &\tr \bm{u}_\gamma \cdot \bm{\nu} = \overline{u_\gamma} &&\text{on } \partial
        \gamma_{u}
    \end{aligned},
\end{gather}
where $\nabla_\gamma \cdot$ is the (inter-layer divergence) conservation operator
{accounting also for the exchange between the damage zone and the fault}.
In \eqref{2d}  $\alpha_\gamma^2$ is the inverse of the
effective matrix permeability in the fault, $q_\gamma$ represents a scalar source or sink term,
$\overline{p_\gamma}$ and $\overline{u_\gamma}$ are the natural and
essential boundary data, respectively.
The conservation operators in each domain are defined as
\begin{gather}\label{eq:gen_divergence}
    \nabla_{\Omega} \cdot \bm{u} \defeq \nabla \cdot \bm{u}_{\Omega},\\
    \nabla_{\mu} \cdot \bm{u} \defeq \nabla \cdot
    \bm{u}_{\mu} - \tr \bm{u}_{\Omega} \cdot \bm{n} + u_{\Gamma},\\
    \nabla_\gamma \cdot \bm{u} \defeq \nabla \cdot \bm{u}_\gamma -
    u_{\Gamma},
\end{gather}
{where} all the differential {(divergence)} operators are defined on the tangent space of the
associated manifold. {Moreover, these conservation operators account for the exchange terms between manifolds of, possibly, different dimensionality.}
To couple the problems we introduce the following interface laws which can be
interpreted as {projections of the Darcy law in the direction normal to the interfaces.}
\begin{align}\label{eq:model_cc}
    \begin{aligned}
        & \alpha_M^2 \tr \bm{u}_{\Omega} \cdot \bm{n} + p_{\mu} - \tr p_{\Omega} = 0
        &&\text{on } M \\
        &\alpha_\Gamma^2 u_{\Gamma} + p_\gamma - p_{\mu} = 0
        &&\text{on } \Gamma
    \end{aligned}.
\end{align}
\end{subequations}
In the previous equations the data are: $\alpha_M^2$ which is the inverse of the
effective normal matrix permeability between the rock matrix and the damage
zone and $\alpha_\Gamma^2$ which is the inverse of the
effective normal matrix permeability between the damage
zone and the fault.
\begin{problem}[Mixed-dimensional model problem]\label{pb:model}
    Find $(\bm{u}, p)$ such that
    \eqref{eq:model} is satisfied.
\end{problem}
\begin{remark}
    In the mixed-dimensional Problem \ref{pb:model} {we have used directly the effective permeabilities, i.e.}
    {the permeabilities already }scaled by the thickness of
    the related object. {They can be related with the parameters of the
    equi-dimensional Problem
    \ref{pb:model_equi} as follows}
    \begin{gather*}
        \sigma_\Omega = \alpha_\Omega, \qquad
        \sigma^2_\mu = \alpha^2_M \epsilon_\mu \bm{n} \otimes \bm{n} +
        \dfrac{\alpha^2_\mu}{\epsilon_\mu}
        (I - \bm{n} \otimes \bm{n}),\qquad
        \sigma^2_\gamma = \alpha^2_\Gamma \epsilon_\gamma \bm{n}_\mu \otimes \bm{n}_\mu +
        \dfrac{\alpha^2_\gamma}{\epsilon_\gamma}
        (I - \bm{n}_\mu \otimes \bm{n}_\mu).
    \end{gather*}
\end{remark}

\section{Weak problem}\label{sec:weak_formulation}

In this section we introduce the functional setting and the weak formulation of
Problem \ref{pb:model}. For simplicity, we consider null flux {essential} boundary conditions
which, through a lifting technique, can be generalized.
In the forthcoming parts we indicate with $\norm{\cdot}_A$ the $L^2$-norm on the
set $A$ and with $(\cdot, \cdot)_A$ the $L^2$-scalar product on the set A.

\subsection{Functional space setting}

Since the previous problem couples unknowns in all the domains, it is natural to
introduce a global functional space to respect this structure. The key
motivation is related to the mixed-dimensional divergence {operators, (\ref{eq:gen_divergence}) that are related to a}
space that generalise the classical (weighted) $H_{\nabla \cdot}-$space in this
framework. For the vector fields, {i.e. $\bm{u}=(\bm{u}_{\Omega}, \bm{u}_{\mu}, \bm{u}_\gamma, u_{\Gamma})$,} we introduce the following
\begin{gather}\label{eq:def_V}
    \begin{aligned}
        V \defeq \{ &\bm{v} = (\bm{v}_\Omega, \bm{v}_\mu, \bm{v}_\gamma, v_\Gamma):\\ \,
        &\alpha_\Omega \bm{v}_\Omega \in L^2(\Omega) \wedge \alpha_M \tr
        \bm{v}_{\Omega} \cdot \bm{n} \in L^2(M) \wedge \nabla_\Omega \cdot
        \bm{v} \in L^2(\Omega) \wedge
        \tr \bm{v}_\Omega \cdot
        \bm{\nu} = 0 \text{ on } \partial \Omega_u,\\
        &\alpha_\mu \bm{v}_\mu \in L^2(\mu) \wedge \nabla_\mu \cdot \bm{v} \in L^2(\mu)
        \wedge \tr \bm{v}_\mu \cdot \bm{\nu} = 0 \text{ on } \partial \mu_u,\\
        &\alpha_{\gamma} \bm{v}_ {\gamma}
        \in L^2(\gamma) \wedge \nabla_{\gamma} \cdot \bm{v} \in L^2(\gamma)
        \wedge \tr \bm{v}_\gamma \cdot \bm{\nu} = 0 \text{ on } \partial
        \gamma_u,\\
        & \alpha_\Gamma v_\Gamma \in L^2(\Gamma) \},
    \end{aligned}
\end{gather}
with the associated norm that make the space $V$ complete, defined as
\begin{gather*}
    \begin{aligned}
        \norm{\bm{v}}_V^2\defeq &\norm{\alpha_\Omega \bm{v}_\Omega}_{\Omega}^2 +
        \norm{\nabla_\Omega \cdot \bm{v}}_{\Omega}^2 + \norm{\alpha_M \tr \bm{v}_\Omega \cdot
        \bm{n}}_{M}^2+
        \norm{\alpha_{\mu} \bm{v}_{\mu}}_{\mu}^2 +
        \norm{\nabla_{\mu} \cdot \bm{v}}_{\mu}^2 +\\
        & \norm{\alpha_{\gamma} \bm{v}_{\gamma}}_{\gamma}^2 +
        \norm{\nabla_\gamma \cdot \bm{v}}_{\gamma}^2+
        \norm{\alpha_\Gamma v_\Gamma}_{\Gamma}^2.
    \end{aligned}
\end{gather*}
We assume that exists $c_\Omega>0$ such that $c_\Omega \leq \alpha_\Omega$ a.e.,
and $\alpha_\Omega \in
L^\infty(\Omega)$. Similarly, we assume that there exists $c_M>0$ such that $c_M \leq
\alpha_M$ a.e., and $\alpha_M \in L^\infty(M)$.
For the damage zone, we require that exists $c_\mu>0$ such that $c_\mu \leq \alpha_{\mu}$
a.e., and also $\alpha_{\mu} \in L^\infty(\mu)$.
Again, we assume that exists $c_\gamma>0$ such that $c_\gamma \leq
\alpha_\gamma$ a.e.,
 and also $\alpha_\gamma \in L^\infty(\gamma)$. Finally, we require
that exists $c_\Gamma > 0$ such that $c_\Gamma \leq \alpha_\Gamma$ a.e.,
and also $\alpha_\Gamma \in L^\infty(\Gamma)$. The extra regularity {required}
for the trace of $\bm{v}_\Omega$ on $M$ is related to {the interface conditions on $M$ which can be seen as a} Robin-type boundary
conditions for a problem in mixed form. For more details see \cite{Roberts1991,Martin2005}.

The functional space for the pressure is defined again on the compound, namely
\begin{gather*}
    Q \defeq \{ v = (v_\Omega, v_\mu, v_\gamma): v_\Omega \in L^2(\Omega), v_\mu
    \in L^2(\mu), v_\gamma \in L^2(\gamma) \},
\end{gather*}
with associated norm that make the space $Q$ complete, given by
\begin{gather*}
    \norm{v}_Q^2 \defeq \norm{v_\Omega}_\Omega^2 + \norm{v_\mu}_\mu^2 +
    \norm{v_\gamma}_\gamma^2.
\end{gather*}

\begin{remark}[Boundary conditions]
    The conditions related to the boundaries in {the definition of } \eqref{eq:def_V} have to be
    interpreted in a proper way. The first condition $\tr \bm{v}_\Omega \cdot\bm{\nu} = 0$
    on $\partial \Omega_u$  {corresponds to}
    \begin{gather*}
        \langle \tr \bm{v}_\Omega \cdot \bm{\nu}, w \rangle_{\partial \Omega_u} = 0
        \quad \text{for all} \quad
        w \in H^{\frac{1}{2}}_{0,0}(\partial \Omega_u)
    \end{gather*}
    with $\langle \cdot, \cdot \rangle_{A}$ the duality pairing defined as
    $\langle \cdot, \cdot \rangle_{A}: H^{-\frac{1}{2}}(A) \times
    H^{\frac{1}{2}}_{0,0}(A) \rightarrow \mathbb{R}$, with $A$ a generic set.
    From standard results we have $\tr \bm{v}_\Omega \cdot \bm{\nu}\in
    H^{-\frac{1}{2}}(\partial \Omega_u)$. In a similar way, the condition $\tr
    \bm{v}_\mu \cdot \bm{\nu} = 0$ on $\partial \mu_u$ means
    \begin{gather*}
        \langle \tr \bm{v}_\mu \cdot \bm{\nu}, w \rangle_{\partial \mu_u} = 0
        \quad \text{for all} \quad
        w \in H^{\frac{1}{2}}_{0,0}(\partial \mu_u).
    \end{gather*}
    In this case, we have $\tr \bm{v}_\mu \cdot \bm{\nu} \in H^{-\frac{1}{2}}(\partial \mu_u)$.
    Finally, the condition  $\tr \bm{v}_\gamma \cdot \bm{\nu} = 0$ on
    $\partial\gamma_u$ is interpreted similarly as
    \begin{gather*}
        \langle \tr \bm{v}_\gamma \cdot \bm{\nu}, w \rangle_{\partial \gamma_u} = 0
        \quad \text{for all} \quad
        w \in H^{\frac{1}{2}}_{0,0}(\partial \gamma_u).
    \end{gather*}
    In this case we have $\tr \bm{v}_\gamma \cdot \bm{\nu} \in H^{-\frac{1}{2}}(\partial \gamma_u)$.
\end{remark}

\begin{remark}
    A more rigorous approach requires the introduction of interpolation
    operators between $\partial \Omega_{\rm int}$ and $M$, as well as between $\mu$ and
    $\Gamma$ and also between $\gamma$ and $\Gamma$. However, to simplify the
    presentation we consider them implicit.
\end{remark}

\subsection{Weak formulation}

We propose now the weak formulation of the previous problem. Again, we respect
the mixed-dimensional nature of Problem \ref{pb:model} by introducing bilinear forms that
consider variables defined on different domains. We first introduce the weighted
$H_{\nabla \cdot}-$mass bilinear form, indicated by $\alpha$ and defined
as $\alpha: V \times V \rightarrow
\mathbb{R}$ such that
\begin{align*}
    \alpha (\bm{u}, \bm{v}) \defeq &
    (\alpha_\Omega \bm{u}_\Omega, \alpha_\Omega \bm{v}_\Omega)_{\Omega} +
    ( \alpha_{M} \tr \bm{u}_\Omega \cdot \bm{n},
    \alpha_M \tr \bm{v}_\Omega \cdot \bm{n})_{M} +(\alpha_{\mu}
    \bm{u}_{\mu}, \alpha_{\mu} \bm{v}_{\mu})_{\mu} +\\
    &
    (\alpha_\gamma \bm{u}_\gamma, \alpha_\gamma \bm{v}_\gamma)_\gamma+
    (\alpha_\Gamma u_\Gamma, \alpha_\Gamma
    v_\Gamma)_\Gamma.
\end{align*}
The bilinear form associated to the conservation statement is given by: $\beta: Q
\times V \rightarrow \mathbb{R}$ such that
\begin{gather*}
    \begin{aligned}
        \beta(p, \bm{v}) \defeq &
        -(p_\Omega, \nabla_\Omega \cdot \bm{v})_{\Omega}
        -(p_{\mu}, \nabla_\mu \cdot \bm{v})_{\mu}
        - (p_\gamma, \nabla_\gamma \cdot \bm{v})_\gamma \\
        =&-(p_\Omega, \nabla \cdot \bm{v}_\Omega)_{\Omega}
        +(p_{\mu}, \tr \bm{v}_\Omega \cdot \bm{n})_{M}
        -(p_{\mu}, \nabla \cdot \bm{v}_\mu)_{\mu}
        - (p_\gamma, \nabla \cdot \bm{v}_\gamma)_\gamma-\\
        &(p_{\mu}, v_\Gamma)_{\Gamma}
        + (p_\gamma, v_\Gamma)_\Gamma.
    \end{aligned}
\end{gather*}
The functionals, which contain the boundary data and source terms, are defined
as: $G: V \rightarrow \mathbb{R}$ and $F:Q \rightarrow \mathbb{R}$ such that
\begin{gather*}
    G(\bm{v}) \defeq -\langle \overline{p_\Omega}, \tr \bm{v}_\Omega \cdot \bm{\nu}
    \rangle_{\partial \Omega_p}
    -\langle \overline{p_\mu}, \tr \bm{v}_\mu \cdot \bm{\nu}
    \rangle_{\partial \mu_p}
    -\langle \overline{p_\gamma}, \tr \bm{v}_\gamma \cdot \bm{\nu}
    \rangle_{\partial \gamma_p}\\
    F(v) \defeq (q_\Omega, v_\Omega)_{\Omega} +
    (q_{\mu},
    v_{\mu})_{\mu} + (q_\gamma, v_\gamma)_\gamma.
\end{gather*}
Where we assume that the pressure boundary data are such that $\overline{p_\Omega}
\in H^{\frac{1}{2}}(\partial \Omega_p)$, $\overline{p_\mu} \in
H^{\frac{1}{2}}(\partial \mu_p)$, and $\overline{p_\gamma} \in H^{\frac{1}{2}}(\partial
\gamma_p)$. We require also that the source terms belong to $(q_\Omega, q_\mu,
q_\gamma) \in Q$.
\begin{problem}[Weak formulation]\label{pb:weak}
    The weak formulation of Problem \ref{pb:model} reads: find $(\bm{u}, p)
    \in V \times Q$ such that
    \begin{gather*}
        \begin{aligned}
            &\alpha(\bm{u}, \bm{v}) + \beta(p, \bm{v}) = G(\bm{v}) && \forall
            \bm{v} \in V\\
            &\beta(v, \bm{u}) = F(v) && \forall v \in Q
        \end{aligned}.
    \end{gather*}
\end{problem}

\subsection{Well posedness}\label{subsec:well_posedness}

In the proof of the well posedness of the problem, some parts are inspired by \cite{Formaggia2012,Fumagalli2016a}.

\begin{theorem}
    Problem \ref{pb:weak} is well posed.
\end{theorem}
\begin{proof}
    The bilinear forms and functionals in  Problem \ref{pb:weak} are linear, so we
    first prove their continuity. For simplicity, we assume that the portions of
    the boundary $\partial \Omega_u$, $\partial \mu_u$, and $\partial \gamma_u$
    are empty. {A lifting technique can be used in the general case.}
    By using Cauchy-Schwarz and triangular
    inequalities we get
    \begin{align*}
        \abs{\alpha(\bm{u}, \bm{v})} \leq &\norm{\alpha_\Omega
        \bm{u}_\Omega}_\Omega
        \norm{\alpha_\Omega \bm{v}_\Omega}_\Omega + \norm{\alpha_{M} \tr \bm{u}_\Omega
        \cdot \bm{n}}_M \norm{\alpha_M \tr \bm{v}_\Omega \cdot \bm{n}}_M +
        \norm{\alpha_{\mu}\bm{u}_{\mu}}_\mu \norm{\alpha_{\mu}
        \bm{v}_{\mu}}_\mu+\\
        &\norm{\alpha_\gamma \bm{u}_\gamma}_\gamma
        \norm{\alpha_\gamma \bm{v}_\gamma}_\gamma+\norm{\alpha_\Gamma
        u_\Gamma}_\Gamma\norm{\alpha_\Gamma v_\Gamma}_\Gamma \leq
        \norm{\bm{u}}_V \norm{\bm{v}}_V
    \end{align*}
    as well as for the $\beta$ bilinear form
    \begin{align*}
        \abs{\beta(p, \bm{v})} \leq & \norm{p_\Omega}_\Omega \norm{\nabla_\Omega
        \cdot \bm{v}}_\Omega + \norm{p_\mu}_\mu \norm{\nabla_\mu \cdot
        \bm{v}}_\mu + \norm{p_\gamma}_\gamma \norm{\nabla_\gamma \cdot
        \bm{v}}_\gamma\leq \norm{p}_Q \norm{\bm{v}}_V.
    \end{align*}
    For the functionals, we consider in addition the trace inequality associated to the
    $H_{\nabla \cdot}-$spaces. We obtain
    \begin{align*}
        \abs{G(\bm{v})} \leq &
        \norm{\overline{p_\Omega}}_{H^{\frac{1}{2}}(\partial \Omega_p)}
        \norm{\tr \bm{v}_\Omega \cdot \bm{\nu}}_{H^{-\frac{1}{2}}(\partial
        \Omega_p)}+\norm{\overline{p_\mu}}_{H^{\frac{1}{2}}(\partial \mu_p)}
        \norm{\tr \bm{v}_\mu \cdot
        \bm{\nu}}_{H^{-\frac{1}{2}}(\partial\mu_p)}+\\
        &\norm{\overline{p_\gamma}}_{H^{\frac{1}{2}}(\partial \gamma_p)}
        \norm{\tr \bm{v}_\gamma \cdot \bm{\nu}}_{H^{-\frac{1}{2}}(\partial
        \gamma_p)} \leq c_G \norm{\bm{v}}_V\\
        \abs{F(v)} \leq & \norm{q_\Omega}_\Omega \norm{v_\Omega}_\Omega+
        \norm{q_\mu}_\mu \norm{v_\mu}_\mu+
        \norm{q_\gamma}_\gamma \norm{v_\gamma}_\gamma \leq c_F \norm{v}_Q,
    \end{align*}
    with the constants
    \begin{gather*}
        c_G = \max\left\{\norm{\overline{p_\Omega}}_{H^{\frac{1}{2}}(\partial \Omega_p)},
        \norm{\overline{p_\mu}}_{H^{\frac{1}{2}}(\partial \mu_p)},
        \norm{\overline{p_\gamma}}_{H^{\frac{1}{2}}(\partial
        \gamma_p)}\right\},\\
        c_F = \max\left\{ \norm{q_\Omega}_\Omega, \norm{q_\mu}_\mu,
        \norm{q_\gamma}_\gamma\right\}.
    \end{gather*}
    Now, we move on to proving the coercivity of $\alpha$ on the kernel of
    $\beta$.
    Considering a function $\bm{w} \in V$ such that $\beta(v, \bm{w}) = 0$ for
    all $v \in Q$, we have by a particular choice of $v$ the following results
    \begin{align*}
        &v = (v_\Omega, 0, 0) &&\Rightarrow&& (v_\Omega, \nabla_\Omega \cdot
        \bm{w})_\Omega
        = 0 \quad \forall v_\Omega \in L^2(\Omega) &&\Rightarrow&& \nabla_\Omega \cdot \bm{w} = 0 \text{ a.e. in }
        L^2(\Omega)\\
        &v = (0, v_\mu, 0) &&\Rightarrow&& (v_\mu, \nabla_\mu \cdot \bm{w})_\mu
        = 0 \quad \forall v_\mu \in L^2(\mu) &&\Rightarrow&& \nabla_\mu \cdot \bm{w} = 0 \text{ a.e. in }
        L^2(\mu)\\
        &v = (0, 0, v_\gamma) &&\Rightarrow&& (v_\gamma, \nabla_\gamma \cdot
        \bm{w})_\gamma
        = 0 \quad \forall v_\gamma \in L^2(\gamma) &&\Rightarrow&& \nabla_\gamma \cdot \bm{w} = 0 \text{ a.e. in }
        L^2(\gamma)
    \end{align*}
    thus the norm of $\bm{w}$ simplifies to
    \begin{gather*}
        \norm{\bm{w}}_V^2 = \norm{\alpha_\Omega \bm{v}_\Omega}_{\Omega}^2+
        \norm{\alpha_M \tr \bm{v}_\Omega \cdot\bm{n}}_{M}^2+
        \norm{\alpha_{\mu} \bm{v}_{\mu}}_{\mu}^2 +
        \norm{\alpha_{\gamma} \bm{v}_{\gamma}}_{\gamma}^2 +
        \norm{\alpha_\Gamma v_\Gamma}_{\Gamma}^2.
    \end{gather*}
    We can prove the coercivity of $\alpha$ on the kernel of $\beta$, by simply
    notice that $\alpha(\bm{w}, \bm{w}) = \norm{\bm{w}}_V^2$.

    To prove the inf-sup condition, given a function $v \in Q$, we introduce the following auxiliary problems
    \begin{gather*}
        \begin{aligned}
            &-\nabla \cdot \nabla \varphi_\Omega = v_\Omega && \text{in } \Omega\\
            &-\tr \nabla \varphi_\Omega \cdot \bm{n} = v_\mu && \text{on } M
            \\
            &\tr\varphi_\Omega = 0 && \text{on } M\\
            &\tr \varphi_\Omega = 0 && \text{on } \partial \Omega_p
        \end{aligned}
        \qquad\qquad
        \begin{aligned}
            &- \nabla \cdot \nabla \varphi_\mu = v_\mu + v_\gamma && \text{in }
            \mu\\
            &\varphi_\mu = 0 && \text{on } \partial \mu_p\\
            &- \nabla \cdot \nabla \varphi_\gamma = - v_\mu && \text{in }
            \gamma\\
            &\varphi_\gamma = 0 && \text{on } \partial \gamma_p
        \end{aligned}
    \end{gather*}
    while for the intersection $\Gamma$ we have $\varphi_\Gamma = v_\mu +
    v_\gamma$. By assuming that the domains are regular enough, from \cite{Dziuk1988} an elliptic regularity
    result can be used giving
    \begin{gather*}
        \varphi_\Omega \in H^2_*(\Omega)
        \quad \text{with} \quad
        \norm{\varphi_\Omega}_{H^2_*(\Omega)} \leq \norm{v_\Omega}_\Omega +
        \norm{v_\mu}_\mu,\\
        \varphi_\mu \in H^2(\mu)
        \quad \text{with} \quad
        \norm{\varphi_\mu}_{H^2(\mu)} \leq \norm{v_\mu}_\mu +
        \norm{v_\gamma}_\gamma,\\
        \varphi_\gamma \in H^2(\gamma)
        \quad \text{with} \quad
        \norm{\varphi_\gamma}_{H^2(\gamma)} \leq
        \norm{v_\mu}_\mu.
    \end{gather*}
    The space $H^2_*(\Omega)$ is the broken $H^2-$space defined on each
    connected component of $\Omega$ (e.g. the left and right part in Figure \ref{fig:domain}), its norm is defined coherently. We clearly
    have $\norm{\varphi_\Gamma}_\Gamma \leq \norm{v_\mu}_\mu +
    \norm{v_\gamma}_\gamma$.
    By considering the function $\bm{w} \in V$ such that
    \begin{gather*}
        \bm{w} = (\bm{w}_\Omega, \bm{w}_\mu, \bm{w}_\gamma, {w}_\Gamma) =
        (\nabla \varphi_\Omega, \nabla \varphi_\mu, \nabla \varphi_\gamma,
        \varphi_\Gamma),
    \end{gather*}
    we obtain that $- \nabla \cdot \bm{w}_\Omega = v_\Omega$
    and $\tr \bm{w}_\Omega \cdot \bm{n} = v_\mu$. For $\bm{w}_\mu$ and  $\bm{w}_\gamma$ we get
    that $- \nabla \cdot \bm{w}_\mu = v_\mu + v_\gamma$
    and $-\nabla \cdot \bm{w}_\gamma = - v_\mu$, respectively. Finally, we have $w_\Gamma =
    v_\gamma + v_\mu$. This gives the following expressions for the
    mixed-dimensional divergences $\nabla_\Omega \cdot \bm{w} = -v_\Omega$,
    $\nabla_\mu \cdot \bm{w} = -v_\mu$, and $\nabla_\gamma \cdot \bm{w} =-
    v_\gamma$.
    The $V-$norm of $\bm{w}$ can be bound by the
    $Q-$norm of $v$, in fact
    \begin{gather*}
        \norm{\bm{w}}_V^2 = \norm{\nabla \varphi_\Omega}^2_\Omega + \norm{\nabla
        \varphi_\mu}_\mu^2 + \norm{\nabla \varphi_\gamma}_\gamma^2 +
        \norm{\varphi_\Gamma}_\Gamma^2 + \norm{v_\Omega}_\Omega^2 +
        2\norm{v_\mu}_\mu^2+\norm{v_\gamma}_\gamma^2 \leq \norm{v}_Q^2.
    \end{gather*}
    Finally we obtain the boundedness of $\beta$ from below with this choice of
    $\bm{w}$, namely
    \begin{gather*}
        \beta(v, \bm{w}) = \norm{v_\Omega}_\Omega^2 + \norm{v_\mu}_\mu^2 +
        \norm{v_\gamma}_\gamma^2 = \norm{v}_Q^2 \geq \norm{v}_Q \norm{\bm{w}}_V.
    \end{gather*}
    Thus the inf-sup condition is fulfilled, and following \cite{Brezzi1991} we conclude that
    Problem \ref{pb:weak} is well posed.
\end{proof}

\section{Numerical discretization}\label{sec:numerical_discretization}

To keep the mixed nature of Problem \ref{pb:weak} in the numerical
approximation, we consider the lowest order Raviart-Thomas-N\'ed\'elec $\mathbb{RT}_0$ finite element for
the Darcy velocity and piecewise constant $\mathbb{P}_0$ for the pressure in the
domains $\Omega$, $\mu$, and $\gamma$. The choice of the pair $\mathbb{RT}_0-\mathbb{P}_0$
is also motivated by their local mass conservation property  and consistency with the
functional spaces considered in the weak formulation.

We introduce a family of simplicial meshes
approximations of the domains $\Omega$, $\mu$, $\gamma$, $\Gamma$
respectively, which will be indicated with the same symbol. By assuming a
matching discretization of the meshes
at the interfaces, the approximation of the functional spaces are
defined as
\begin{gather*}
    V_h \defeq \mathbb{RT}_0(\Omega) \times \mathbb{RT}_0(\mu) \times
    \mathbb{RT}_0(\gamma) \times \mathbb{P}_0(\Gamma) \subset V\\
    Q_h \defeq \mathbb{P}_0(\Omega) \times \mathbb{P}_0(\mu) \times
    \mathbb{P}_0(\gamma) \subset Q,
\end{gather*}
we introduce a set of base functions for the discrete spaces, such that
\begin{gather*}
    \mathbb{RT}_0(\Omega) = \spanop_{i \in \dof
    \mathbb{RT}_0({\Omega})} \{
    \bm{\zeta}_{\Omega, i} \}
    \quad\text{and}\quad
    \mathbb{P}_0(\Omega) = \spanop_{i\in \dof
    \mathbb{P}_0(\Omega)} \{ {\xi}_{\Omega, i}
    \}\\
    \mathbb{RT}_0(\mu) = \spanop_{i\in \dof \mathbb{RT}_0({\mu})} \{ \bm{\zeta}_{\mu, i}
    \}
    \quad\text{and}\quad
    \mathbb{P}_0(\mu) = \spanop_{i \in \dof \mathbb{P}_0(\mu)} \{
    {\xi}_{\mu, i} \}
    \\
    \mathbb{RT}_0(\gamma) = \spanop_{i \in \dof
    \mathbb{RT}_0({\gamma})} \{ \bm{\zeta}_{\gamma, i} \}
    \quad\text{and}\quad
    \mathbb{P}_0(\gamma) = \spanop_{i \in \dof
    \mathbb{P}_0(\gamma)} \{
    {\xi}_{\gamma, i} \}\\
    \mathbb{P}_0(\Gamma) = \spanop_{i \in \dof
    \mathbb{P}_0(\Gamma)} \{
    {\zeta}_{\Gamma, i} \}.
\end{gather*}
We assume that quadrature is performed exactly.
We indicate with $h$ the global mesh size of the discretization.

\subsection{Matrix formulation} \label{subsec:matrix_form}

Given the choice of the discrete spaces, we can recast the weak formulation of
the problem in term of a block matrix system. We introduce the block matrices
related to the mass matrices
\begin{gather*}
    [A_\Omega]_{i,j} \defeq (\alpha_\Omega \bm{\zeta}_{\Omega, i}, \alpha_\Omega
    \bm{\zeta}_{\Omega, j})_{\Omega} +
    ( \alpha_{M} \tr \bm{\zeta}_{\Omega, i} \cdot \bm{n},
    \alpha_M \tr \bm{\zeta}_{\Omega, j} \cdot \bm{n})_{M}\\
    [A_\mu]_{i, j} \defeq (\alpha_{\mu}
    \bm{\zeta}_{\mu, i}, \alpha_{\mu} \bm{\zeta}_{\mu, j})_{\mu}
    \quad
    [A_\gamma]_{i, j } \defeq (\alpha_\gamma \bm{\zeta}_{\gamma, i},
    \alpha_\gamma \bm{\zeta}_{\gamma, j})_\gamma
    \quad
    [A_\Gamma]_{i, j} \defeq (\alpha_\Gamma \zeta_{\Gamma, i}, \alpha_\Gamma
    \zeta_{\Gamma, j})_\Gamma,
\end{gather*}
the matrices associated with the conservation statement in each domain are
defined as
\begin{gather*}
    [B_\Omega]_{i, j} \defeq -(\xi_{\Omega,j} , \nabla \cdot \bm{\zeta}_{\Omega,
    i})_{\Omega}\quad
    [B_\mu]_{i, j} \defeq -(\xi_{\mu, j}, \nabla \cdot \bm{\zeta}_{\mu, i})_{\mu}
    \quad
    [B_\gamma]_{i, j} \defeq - (\xi_{\gamma, j}, \nabla \cdot
    \bm{\zeta}_{\gamma, i})_\gamma,
\end{gather*}
while between the domains the coupling conditions are associated with the
matrices
\begin{gather*}
    [G_\Omega]_{i, j} \defeq - (\xi_{\mu, j}, \tr \bm{\zeta}_{\Omega, i} \cdot \bm{n})_M
    \quad
    [G_\mu]_{i, j} \defeq -(\xi_{\mu, i}, \zeta_{\gamma, j})_\Gamma
    \quad
    [G_\gamma]_{i, j} \defeq (\xi_{\gamma, i}, \zeta_{\gamma, j})_\Gamma.
\end{gather*}
We indicate with $\mathfrak{u}$ and $\mathfrak{p}$ the values of degrees of
freedom associated with the Darcy
velocity and pressure. We make use of a similar notation to indicate the source and
boundary terms. The linear system
associated with Problem \ref{pb:weak} is
\begin{gather}\label{eq:matrix_form}
    \begin{bmatrix}
        A_\Omega & B_\Omega & & G_\Omega& & & \\
        B_\Omega^{\top} & & & & & & \\
        & & A_{\mu} & B_{\mu} & & & \\
        G_\Omega^\top& & B_{\mu}^\top & & & & G_\mu\\
        & & & & A_{\gamma} & B_{\gamma} & \\
        & & & & B_{\gamma}^T & & G_\gamma\\
        & & & G_\mu^\top& & G_\gamma^\top& A_\Gamma
    \end{bmatrix}
    \begin{bmatrix}
        \mathfrak{u}_{\Omega}\\
        \mathfrak{p}_{\Omega}\\
        \mathfrak{u}_{\mu}\\
        \mathfrak{p}_{\mu}\\
        \mathfrak{u}_{\gamma}\\
        \mathfrak{p}_{\gamma}\\
        \mathfrak{u}_{\Gamma}
    \end{bmatrix}
    =
    \begin{bmatrix}
        \mathfrak{g}_\Omega\\
        \mathfrak{f}_\Omega\\
        \mathfrak{g}_\mu\\
        \mathfrak{f}_\mu\\
        \mathfrak{g}_\gamma\\
        \mathfrak{f}_\gamma\\
        \phantom{0}
    \end{bmatrix}.
\end{gather}

where we have avoided explicitly writing empty matrices.
Since the bilinear forms associated to the matrices $A_\Omega$, $A_\mu$,
$A_\gamma$, $A_\Gamma$ are symmetric then also these matrices are symmetric,
resulting in a symmetric global system with a saddle-point structure.
We can introduce the following problem.
\begin{problem}[Matrix formulation]\label{pb:matrix_form}
    The matrix formulation of Problem \ref{pb:weak} is: find $(\mathfrak{u},
    \mathfrak{p})$ such that \eqref{eq:matrix_form} is satisfied.
\end{problem}
To recast
the previous system in terms of the pressures alone, we introduce the matrices
\begin{gather*}
    S_\Omega \defeq B_\Omega^{\top}A_\Omega^{-1} B_\Omega \quad
    S_\mu \defeq B_{\mu}^\top A_{\mu}^{-1} B_{\mu} + G_\Omega^\top A_\Omega^{-1}
    G_\Omega + G_\mu A_\Gamma^{-1} G_\mu^\top\\
    S_\gamma \defeq B_{\gamma}^T A_{\gamma}^{-1} B_{\gamma} + G_\gamma A_\Gamma^{-1}
    G_\gamma^\top\quad
    C_\Omega \defeq B_\Omega^{\top}A_\Omega^{-1}G_\Omega \quad
    C_\mu \defeq G_\mu^\top A_\Gamma^{-1} G_\gamma^\top,
\end{gather*}
where the matrices $A_\Omega$, $A_\mu$, $A_\gamma$, and $A_\Gamma$ are invertible
since they are a Raviart-Thomas-N\'ed\'elec approximation of the $H_{\nabla
\cdot}-$mass bilinear forms and $\partial \Omega_p \neq \emptyset$, $\partial
\mu_p \neq \emptyset$, and $\partial \gamma_p \neq \emptyset$, while $A_\Gamma$
is invertible by construction. The previous matrices are thus well defined
and $S_\Omega$, $S_\mu$, $S_\gamma$ are also symmetric and positive definite.
We introduce the vectors
\begin{gather*}
    \mathfrak{r}_\Omega \defeq - \mathfrak{f}_\Omega + B_\Omega^{\top} A_\Omega^{-1}
    \mathfrak{g}_\Omega\quad
    \mathfrak{r}_\mu \defeq - \mathfrak{f}_\mu + B_{\mu}^\top  A_{\mu}^{-1}
    \mathfrak{g}_\mu + G_\Omega^\top A_\Omega^{-1} \mathfrak{g}_\Omega\\
    \mathfrak{r}_\gamma \defeq - \mathfrak{f}_\gamma +  B_{\gamma}^T
    A_{\gamma}^{-1} \mathfrak{g}_\gamma,
\end{gather*}
which are again well defined. The system in terms of pressure can be written as
\begin{gather} \label{eq:linear_system_pressure}
    \begin{bmatrix}
        S_\Omega & C_\Omega & \\
        C_\Omega^\top & S_\mu & C_\mu \\
        & C_\mu^\top & S_\gamma
    \end{bmatrix}
    \begin{bmatrix}
        \mathfrak{p}_{\Omega}\\
        \mathfrak{p}_{\mu}\\
        \mathfrak{p}_{\gamma}
    \end{bmatrix}
    =
    \begin{bmatrix}
        \mathfrak{r}_\Omega\\
        \mathfrak{r}_\mu\\
        \mathfrak{r}_\gamma
    \end{bmatrix}.
\end{gather}
\begin{problem}[Pressure matrix formulation]
    The pressure matrix formulation of Problem \ref{pb:weak} is: find
    $\mathfrak{p}$ such that \eqref{eq:linear_system_pressure} is satisfied.
\end{problem}
To numerically solve the problem, it is possible to use the equivalent
formulations of Problem
\ref{pb:matrix_form} or Problem \ref{eq:linear_system_pressure}.

\section{Numerical examples}\label{sec:numerical_examples}

In this part we present some numerical examples to show different aspects of the
mathematical model introduced previously. In particular, {we focus on its
behavior in the presence of} high contrasts in permeability among the rock
matrix, damage zone and fault. We present also the model error, {i.e. the
error introduced by the geometrical reduction of the three layers} and discuss
the obtained results. Finally, we show the applicability in a three-dimensional
domain.

All the test cases are {implemented} in the library PorePy \cite{Keilegavlen2017a}.
The current implementation in the code considers a mortar variable on each
interface and it is thus capable to
handle non-matching discretization. Also, the code is agnostic with respect to
the numerical scheme adopted in each domain. However, to keep the presentation
simple and coherent with the previous sections, we will consider matching grids
and  Raviart-Thomas-N\'ed\'elec {finite element of the lowest order for the numerical} approximation.

\subsection{Example 1}\label{subsec:example1}

This first {set of numerical tests} is divided in two parts. In the first we present the effect of
permeability heterogeneity in the fault and damage zone on the solution.
This test is inspired by the examples presented in \cite{Martin2005,Frih2011}.
We study the case of high permeability, low permeability, and a
mixed case. The aim is to present the potentialities of the introduced model in
the Sub-subsection \ref{subsubsec:example1_1}. In
the second part, Sub-subsection \ref{subsubsec:example1_2}, the model error
is studied {comparing the mixed-dimensional solution with the equi-dimensional one where
a full Darcy problem is solved on a computational grid refined enough to resolve the aperture
of both the damage zone and the fault.}

In all the cases we consider a fixed geometry and boundary conditions on the
rock matrix. The rock matrix {occupies} the domain $\Omega = (0, 1) \times (0, 1) \cup
(1, 2) \times (0, 1)$, while the damage zone and fault are identified as $\mu =
\{1\} \times (0, 1)$ and $\gamma = \{1\} \times (0, 1)$, {respectively}. We set $\alpha_\Omega =
1$ and pressure boundary condition on the left and right of $\Omega$, with
values 0 and 1, respectively. On the remaining portions of the boundary we impose
zero flux. The computational domain is represented in Figure \ref{fig:domain}. The grid is
composed of $\sim 8.5k$ triangles for $\Omega$, $80$ segments for $\mu$,
$40$ segments for $\gamma$, and $40$ segments for $\Gamma$.

We consider three different sub-cases, depending on the value of $\alpha$ in the
damage zone and fault. In all the cases, we set $\epsilon = 10^{-2}$. In the
\textit{case (i)} we have $\alpha_\mu^2 = 10^2\epsilon$ and $\alpha_M^2 =
10^2/\epsilon$ for the damage zone, and $\alpha_\gamma^2 = 10^2\epsilon$ and
$\alpha_\Gamma^2 = 10^2/\epsilon$ for the fault. For the \textit{case (ii)}, the
parameters are chosen as $\alpha_\mu^2 = k\epsilon$ and $\alpha_M^2 =
k/\epsilon$ for the damage zone, and $\alpha_\gamma^2 = k\epsilon$ and
$\alpha_\Gamma^2 = k/\epsilon$ for the fault, where $k$ is given by
\begin{gather*}
    k(y) =
    \begin{dcases*}
        1 & if $y<0.25$ or $y>0.75$\\
        2\cdot 10^{-3} & if $0.25 \leq y \leq 0.75$.
    \end{dcases*}
\end{gather*}
Finally, in \textit{case (iii)} we impose the values of $\alpha_\gamma^2 = k
\epsilon$ and
$\alpha_\Gamma^2 = k / \epsilon$ for the fault. The damage zone is divided
into its left and right parts, on the former we have $\alpha_\mu^2 = k \epsilon$
and $\alpha_M^2 = k / \epsilon$ while on the latter $\alpha_\mu^2 = 10^2
\epsilon$ and $\alpha_M^2 = 10^2 / \epsilon$.
In \textit{case (i)} the lower dimensional objects are highly conductive, in
\textit{case (ii)} are heterogeneous in space and less conductive in the central part, and
in \textit{case (iii)} the damage zone is heterogeneous in space and asymmetric.

In \textit{case (i)} we impose unitary pressure at damage zone and fault tips $y=1$,
while zero pressure on the other side. For \textit{case (ii)} and \textit{case
(iii)} zero flux is imposed on the boundary of the damage zone and fault.

\subsubsection{Permeability contrast}\label{subsubsec:example1_1}

In this first section, we present the numerical results obtained from the reduced
model in the three different cases. In Figure \ref{fig:example1_case1} the
results from \textit{case (i)} are shown.
\begin{figure}
    \centering
    \includegraphics[width=0.55\textwidth]{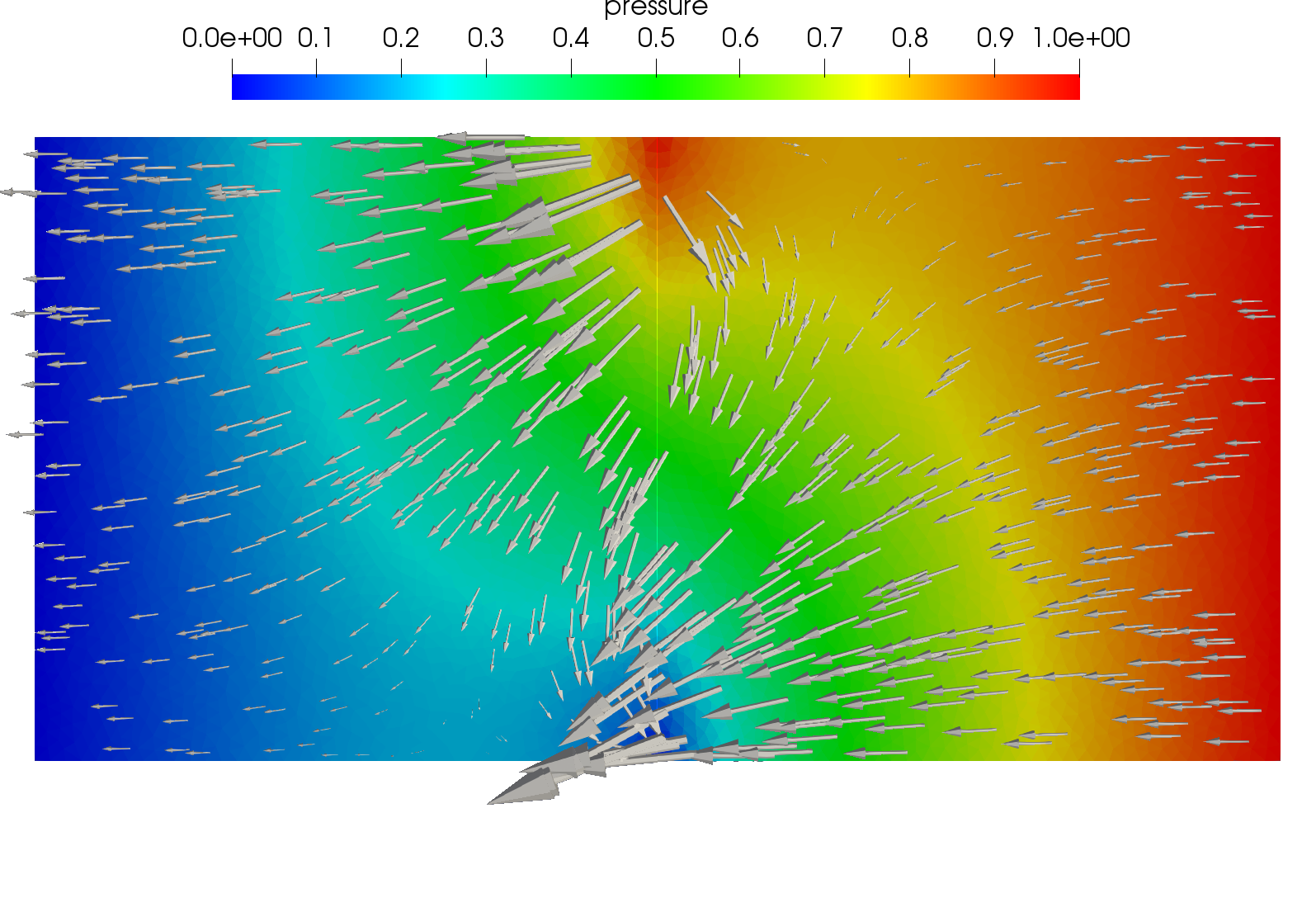}%
    \hspace*{0.1\textwidth}%
    \includegraphics[width=0.25\textwidth]{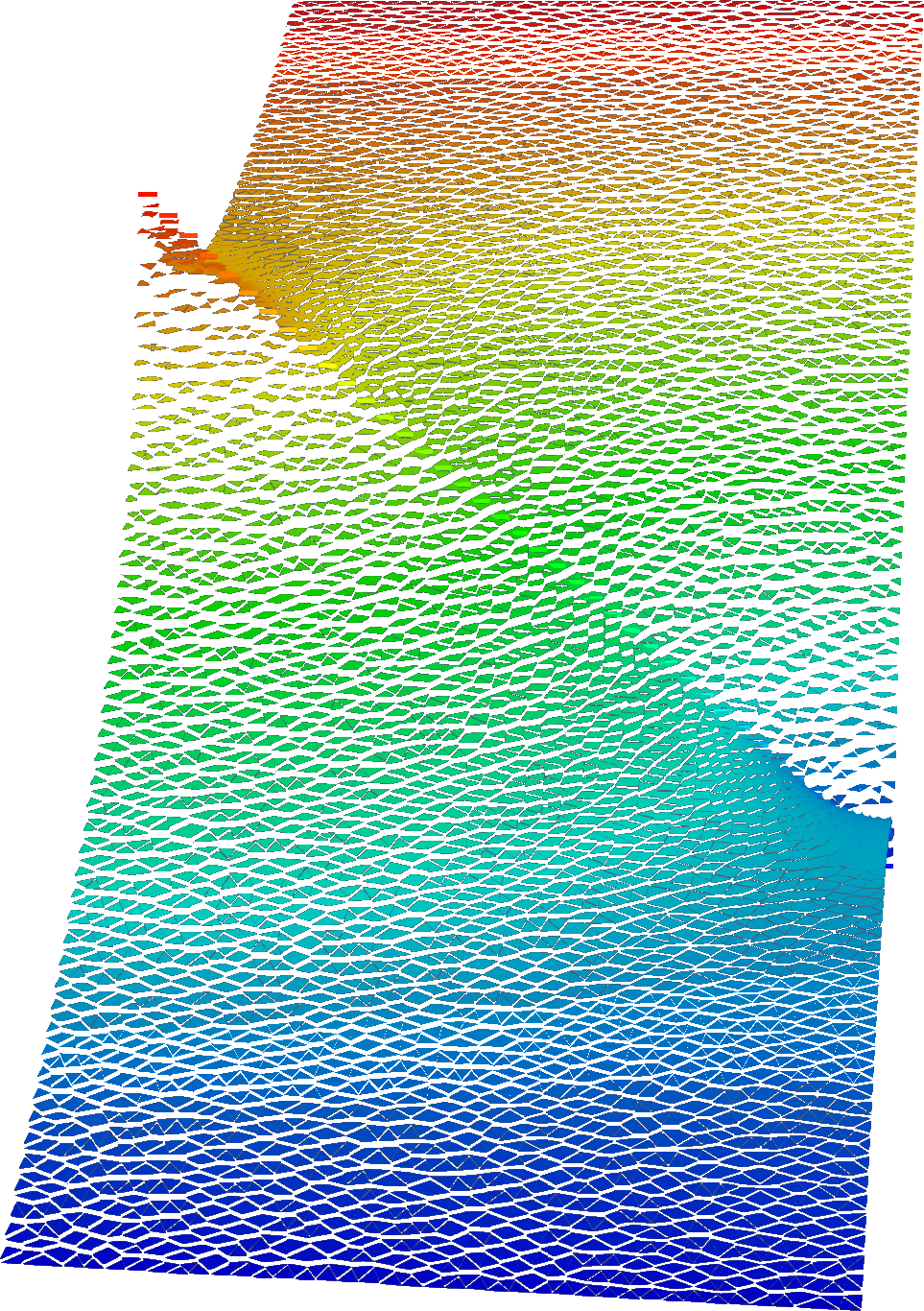}
    \caption{Graphical representation of the solution of \textit{case (i)}. On
    the left the pressure field and the Darcy velocity, the latter shown only
    for few cells. On the right the pressure field warped and rotated.}
    \label{fig:example1_case1}
\end{figure}
Due to the high value of $\alpha_\mu$, $\alpha_M$, $\alpha_\gamma$, and
$\alpha_\Gamma$ the pressure profile is smooth (continuous) across $\mu$ and $\gamma$. The
impact of these lower dimensional objects {on the flow} is still quite {remarkable}, due to the
 type of boundary conditions {and the permeability contrast}.

The solution for the \textit{case (ii)} is given in Figure
\ref{fig:example1_case2}.
\begin{figure}
    \centering
    \includegraphics[width=0.55\textwidth]{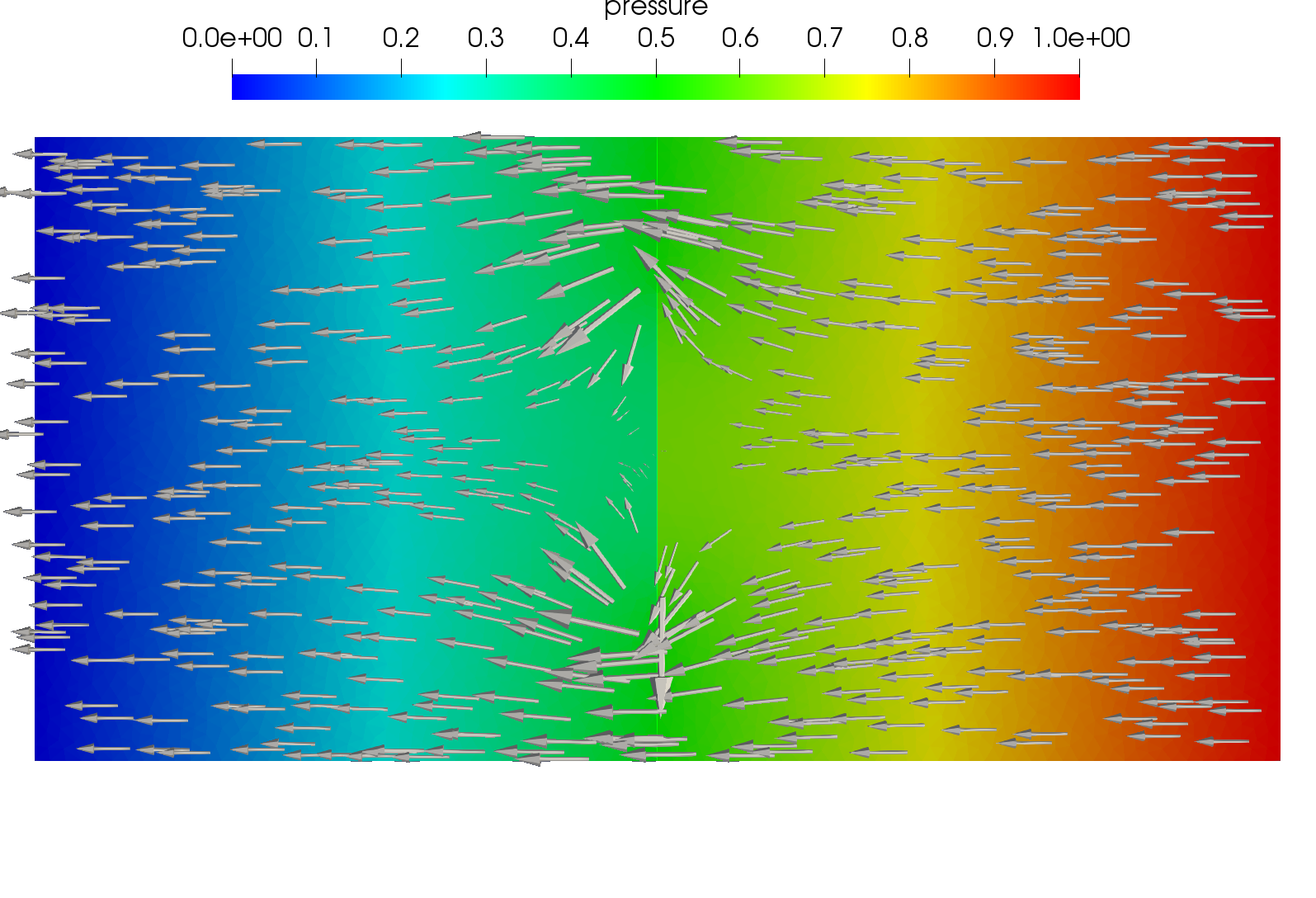}%
    \hspace*{0.1\textwidth}%
    \includegraphics[width=0.25\textwidth]{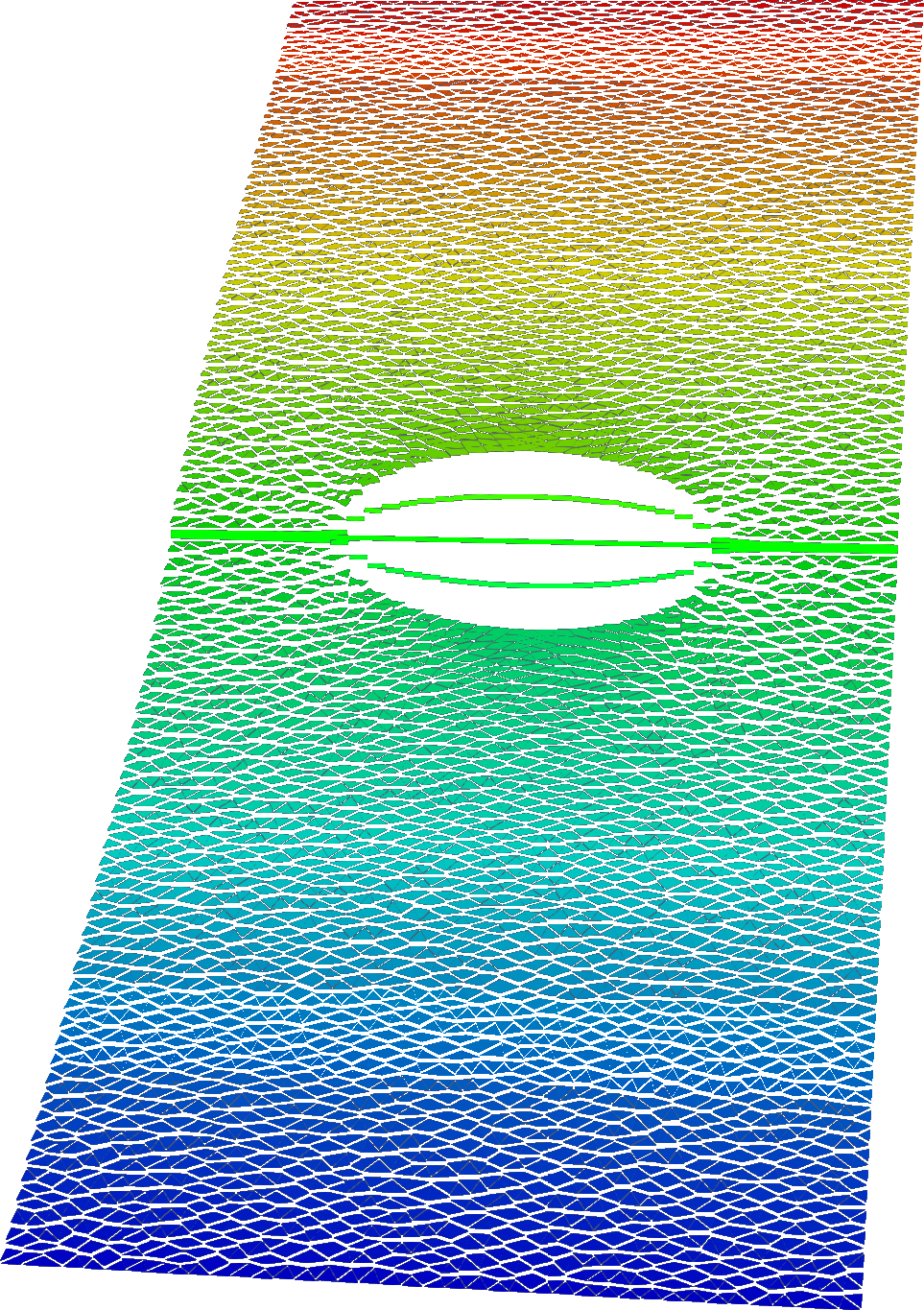}
    \caption{Graphical representation of the solution of \textit{case (ii)}. On
    the left the pressure field and the Darcy velocity, the latter shown only
    for few cells. On the right the pressure field warped and rotated.}
    \label{fig:example1_case2}
\end{figure}
The impact of the low value of the {permeabilities} in the central parts of $\mu$ and
$\gamma$ is evident. We have a pressure jump between the rock matrix and the
damage zone, and between the damage zone and the fault. The pressure in the latter is {constant}
due to the symmetry of the problem. {The flow tends to focus around these less
permeable regions}.

Figure \ref{fig:example1_case3} shows the results from the \textit{case (iii)}.
\begin{figure}
    \centering
    \includegraphics[width=0.55\textwidth]{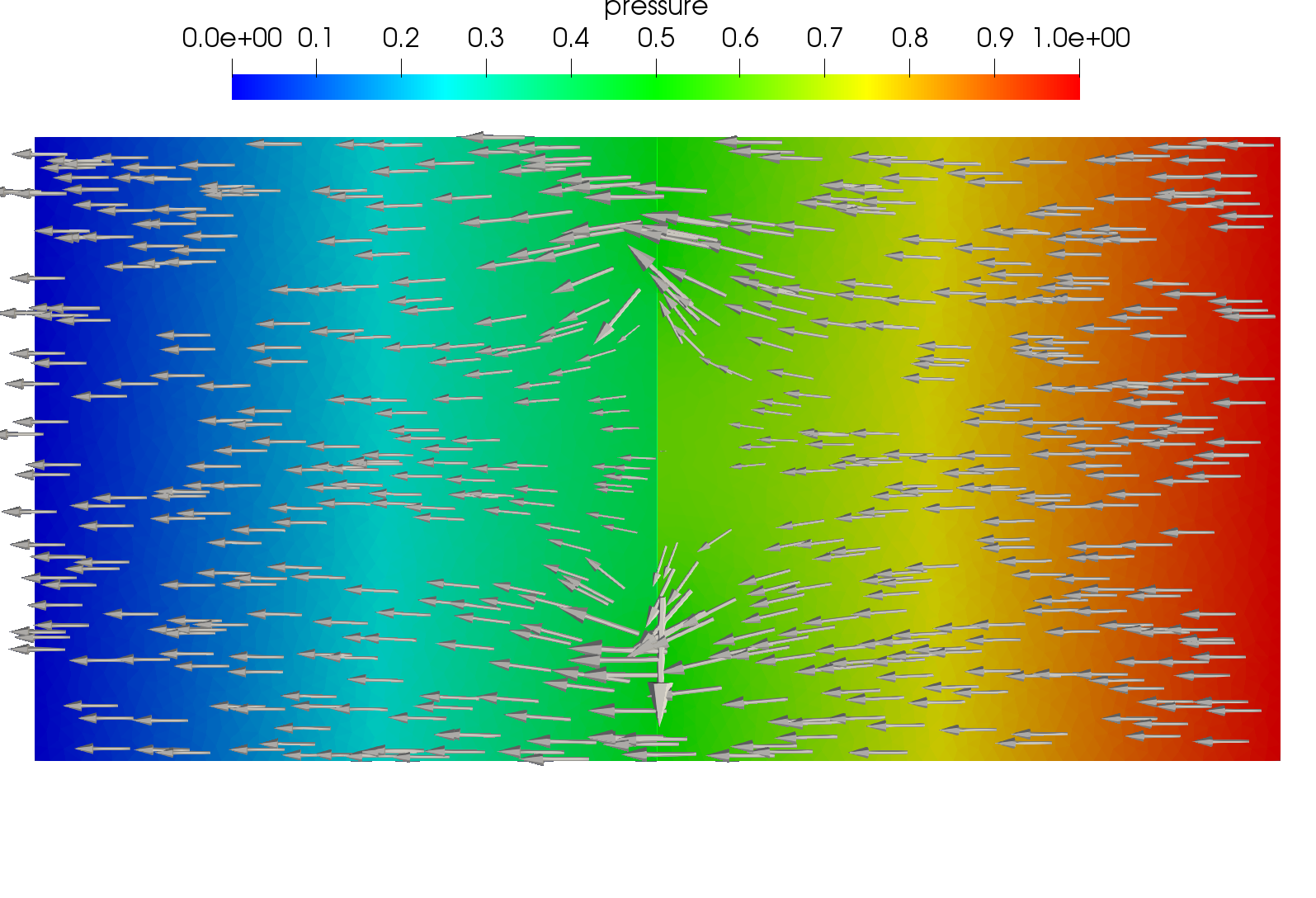}%
    \hspace*{0.1\textwidth}%
    \includegraphics[width=0.25\textwidth]{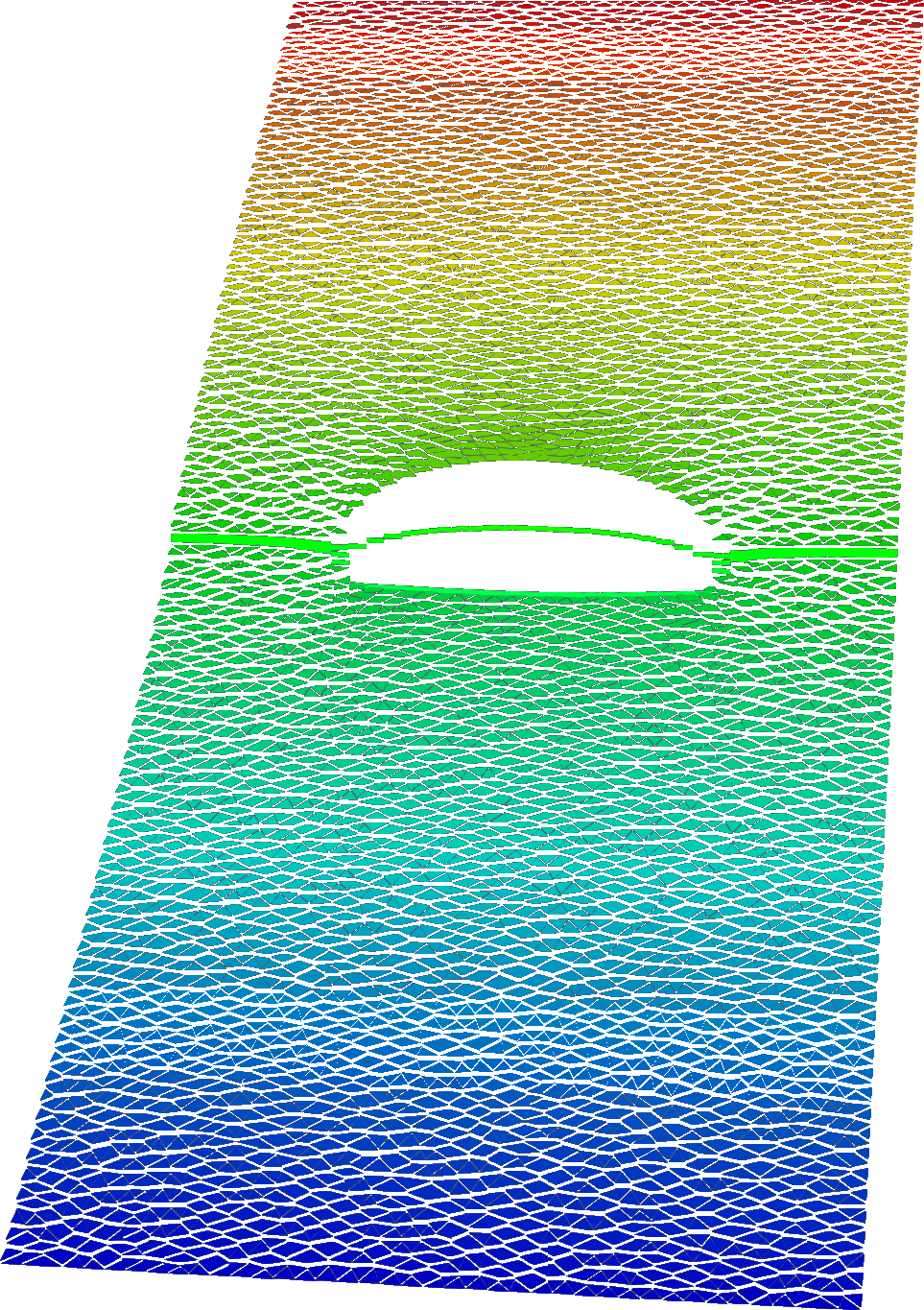}
    \caption{Graphical representation of the solution of \textit{case (iii)}. On
    the left the pressure field and the Darcy velocity, the latter shown only
    for few cells. On the right the pressure field warped and rotated.}
    \label{fig:example1_case3}
\end{figure}
In this test a side of $\mu$ has a portion with low $\alpha_\mu$ and $\alpha_M$,
while on the other side these coefficients have hight values. The pressure solution
exhibit thus a jump between the rock matrix and the first side of $\mu$, and
with the latter and the fault. However, {on the other side} the solution behaves similarly to \textit{case (i)} and we
obtain a smooth (continuous) profile among the fault and the second part of
$\mu$ and the fault. The Darcy velocity tends to avoid the low permeable part of $\mu$,
but not the high permeable one.

{The test cases above demonstrate the capability of the model to handle different combinations of the model parameters.}
The obtained solutions are physically sound and show the capability and
potentiality of the model.

\subsubsection{Model error}\label{subsubsec:example1_2}

In this section we discuss the error associated with the geometrical reduction
of the fault and damage zone, which are modeled as $n-1$ dimensional
interfaces even if, physically, they are $n$ dimensional regions as $\Omega$. To
estimate this error we will compare the solutions provided by the
mixed-dimensional model with the numerical solution of a traditional full Darcy
problem set on a domain with heterogeneous permeabilities, discretized with a
grid that is able to resolve the actual aperture of the layers. The
equi-dimensional solution is computed with the same pair of mixed finite
elements as the mixed-dimensional one. Let $p_\mathrm{equi}^\eta$ be the
numerical solution of the equi-dimensional problem on a grid $\mathcal{T}_\eta$
of size $\eta$. Let $I^\eta(p^h_\Omega)$ be the interpolation of $p^h_\Omega$ on
$\mathcal{T}_\eta$: in Figure \ref{fig:example1_case1_model_error} we show the
difference $I^\eta(p^h_\Omega)-p_\mathrm{equi}^\eta$ for the three cases
presented in the previous section.
\begin{figure}
    \centering
    \includegraphics[width=0.5\textwidth]{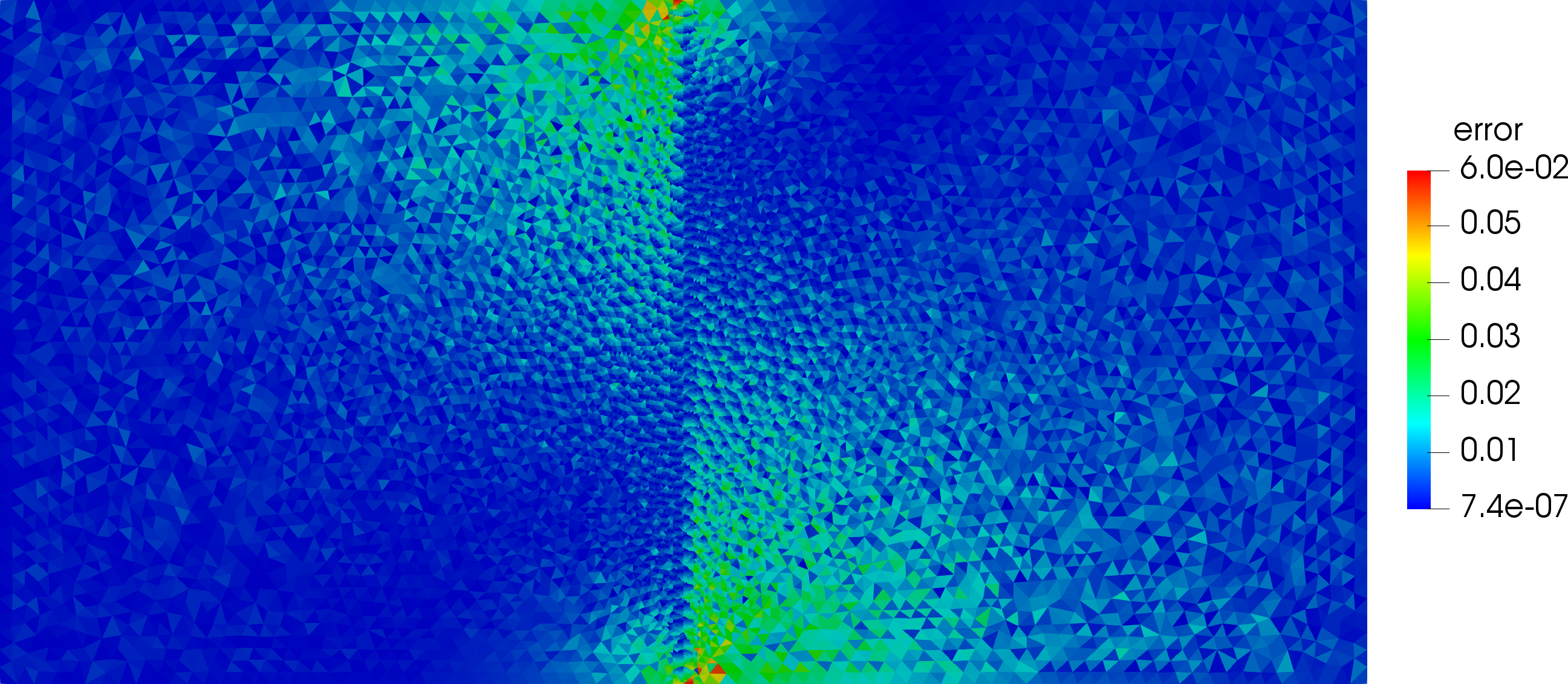}%
    \hspace*{0.1\textwidth}%
    \includegraphics[width=0.39\textwidth]{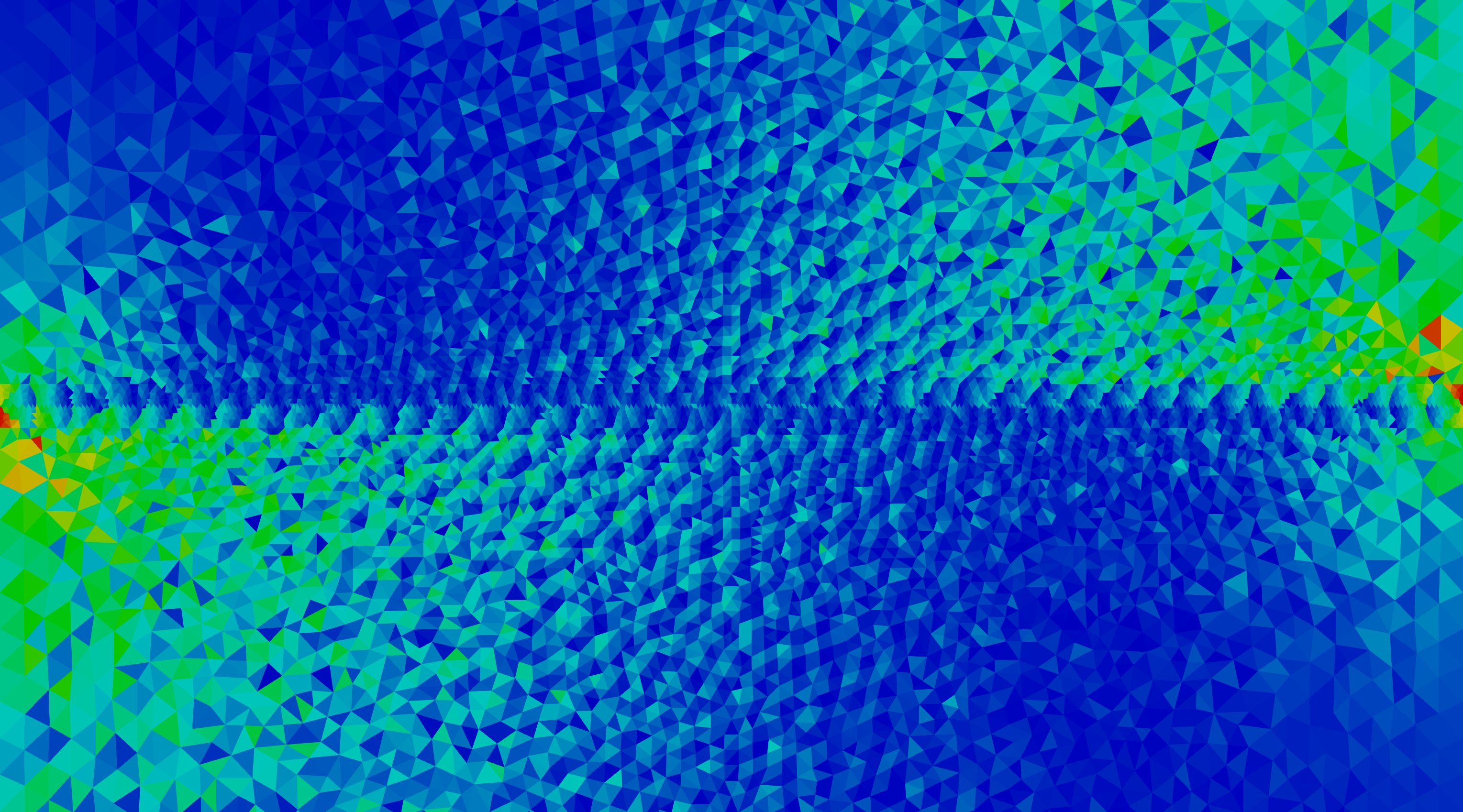}\\
    \vspace*{0.025\textwidth}
    \includegraphics[width=0.5\textwidth]{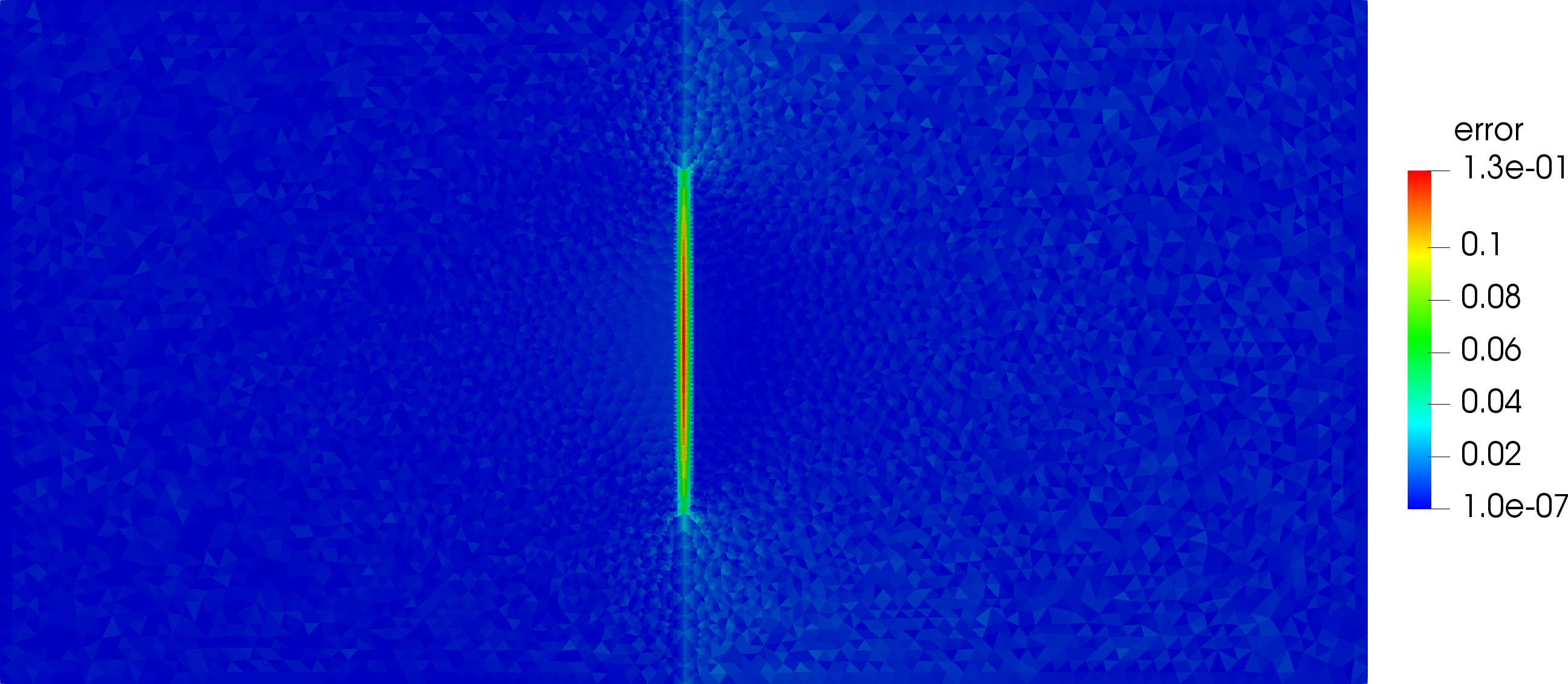}%
    \hspace*{0.1\textwidth}%
    \includegraphics[width=0.39\textwidth]{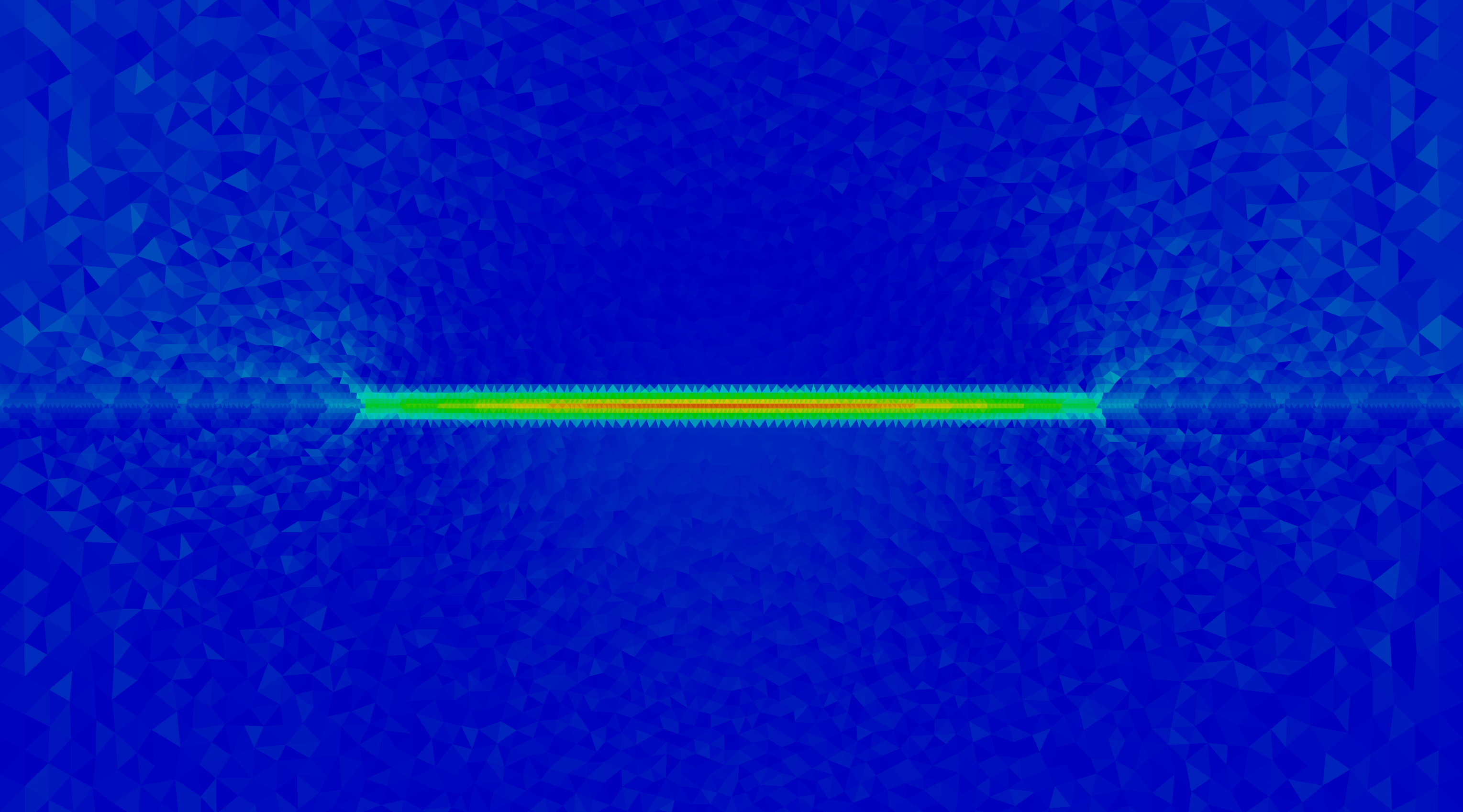}\\
    \vspace*{0.025\textwidth}
    \includegraphics[width=0.5\textwidth]{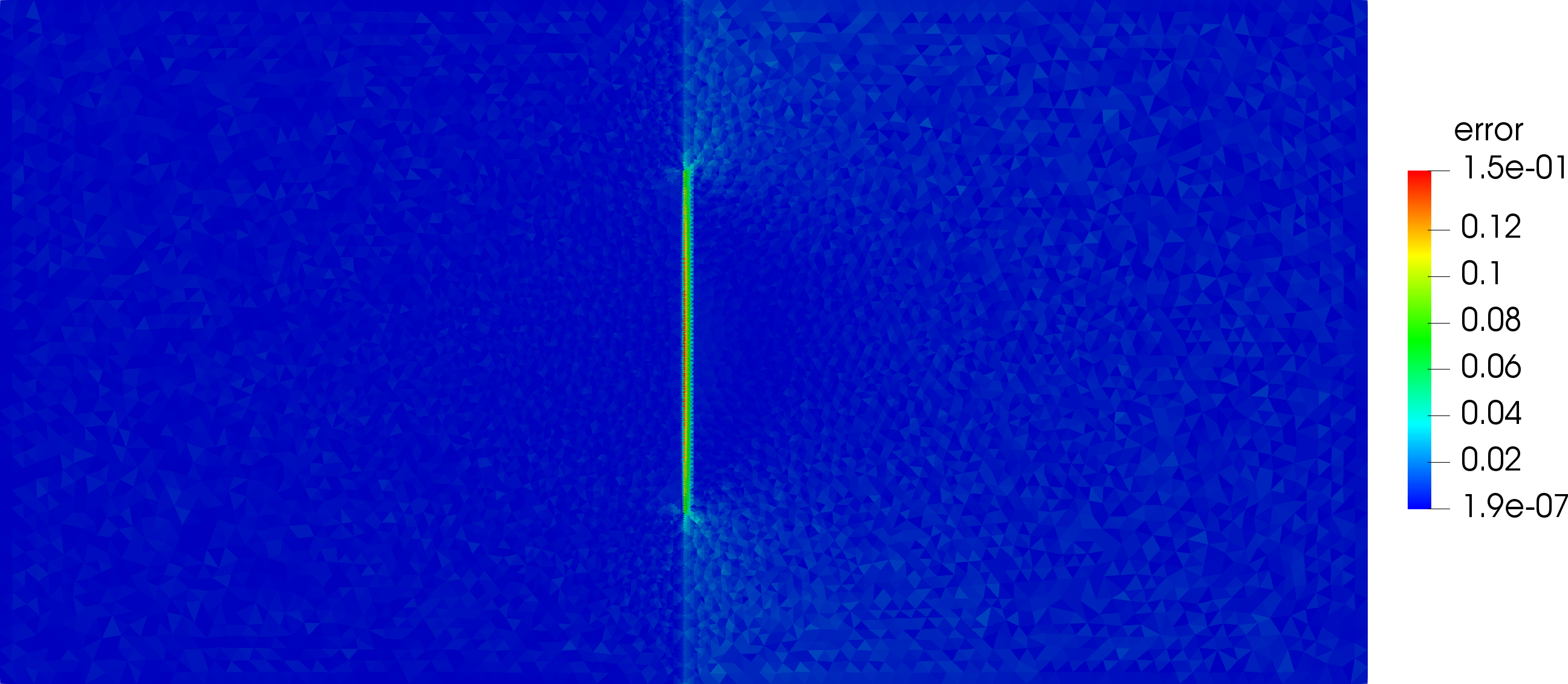}%
    \hspace*{0.1\textwidth}%
    \includegraphics[width=0.39\textwidth]{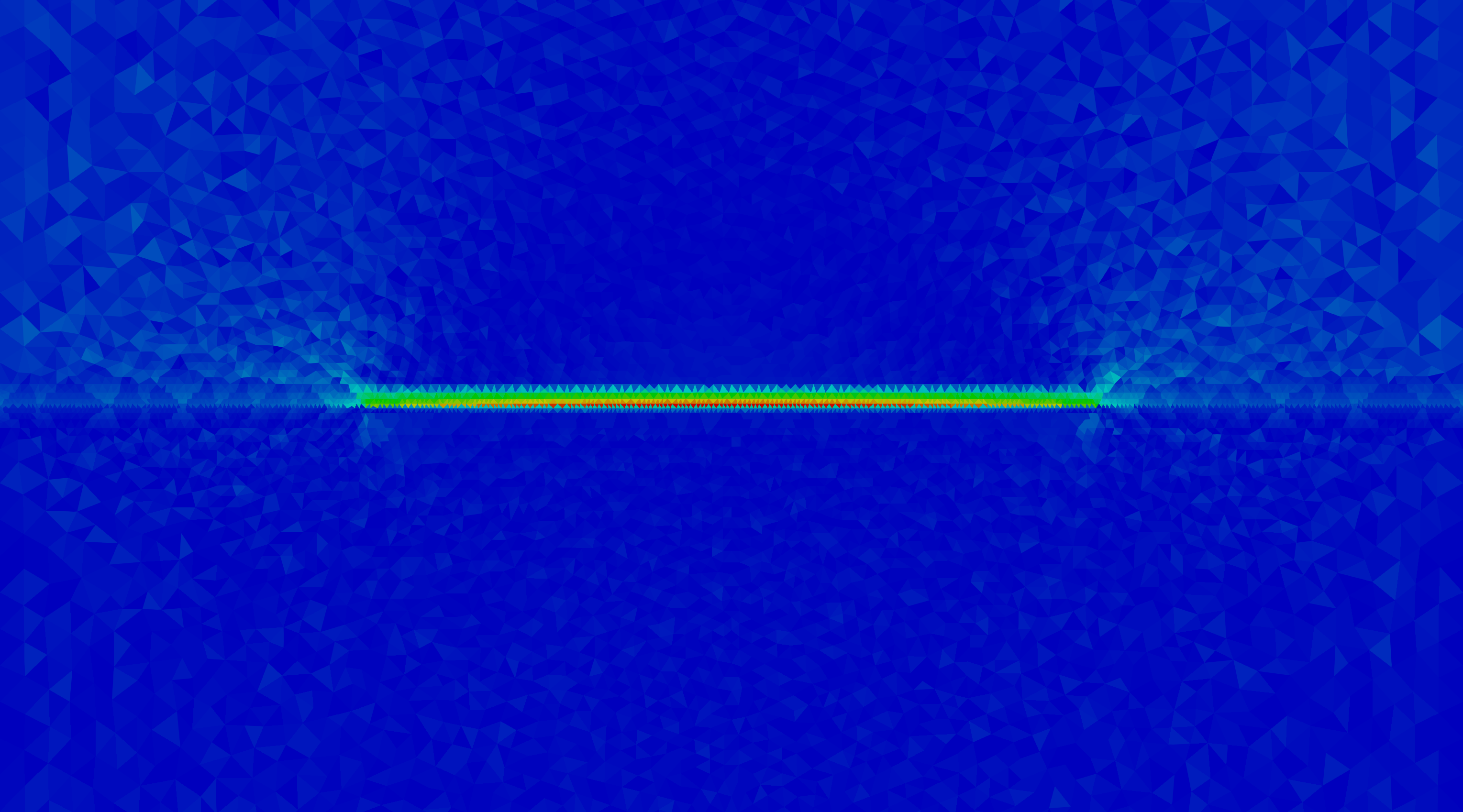}
    \caption{Graphical representation of the solution of \textit{case (iii)}. On
    the left the pressure field and the Darcy velocity, the latter shown only
    for few cells. On the right the pressure field warped and rotated. The pressure is
    scaled in $[0, 1]$, from red to blue.}
    \label{fig:example1_case1_model_error}
\end{figure}

We can observe that, as expected, for the second and third cases the difference
is mostly focused in the fault and surrounding layers: there indeed the
equi-dimensional solution exhibits strong gradients which are replaced by jumps
in the mixed-dimensional approximation due to the shrinking of the layers into
interfaces. It can also be observed the asymmetry in the model error of the
third case, reflecting the different permeabilities of the two layers of the
damage zone.
In the first case however, where the fault is conductive, pressure is
continuous everywhere and the largest values of the error are due to the strong
gradients close to the fault which are not well captured by the coarse mixed-dimensional grid.

For a more quantitative analysis we compute $\tilde{e}_m$ as the $L^2(\Omega)$
norm of $I^\eta(p^h_\Omega)-p_\mathrm{equi}^\eta$ for different apertures
$\epsilon$ of the fault and surrounding layers. If the geometrical reduction is
consistent we expect a reduction of this error with smaller values of
$\epsilon$. The results for the three cases are reported in Table
\ref{tag:example1_error}. Note that as $\epsilon$ decreases we observe a
saturation of the error. This is due to the fact that we are indeed measuring
together the model error and the numerical error, i.e. the error due to a coarse
mixed-dimensional grid.

\begin{table}
    \centering
    \begin{tabular}{|l|c|c|c|}
        \hline
        $\epsilon$          & \textit{case (i)}   & \textit{case (ii)}  & \textit{case (iii)} \\
        \hline
        $10^{-2}$           &      1.12812e-2             &       9.72846e-3              &      8.93785e-3               \\
        \hline
        $5\cdot 10^{-3}$    &      7.29157e-3             &        6.83495e-3            &        6.36322e-3              \\
        \hline
        $2.5 \cdot 10^{-3}$ &       5.7667e-3             &          5.24812e-3           &       5.21005e-3              \\
        \hline
    \end{tabular}
    \caption{Errors corresponding to different values of the $\epsilon$ for the
    three different cases of the test in Subsection \ref{subsec:example1}.}%
    \label{tag:example1_error}
\end{table}

To isolate the two effects we proceed as follows. Assuming that
$p_\mathrm{equi}^\eta$ is fully resolved and can replace the exact
equi-dimensional solution, we define the model error as
\begin{gather*}
    e_m=\| I^\eta(p_\Omega^\mathrm{exact}) -  p_\mathrm{equi}^\eta   \|_{L^2(\Omega)}
\end{gather*}
where $p_\Omega^\mathrm{exact}$ is the fully resolved mixed-dimensional solution, which is in general not available
since the aim of reduced model is to avoid extreme refinement. Let
$p_\Omega^{h_2}$ be the solution for a second mixed-dimensional grid with $h_2<h$, then
\begin{gather*}
    e_m=\| I^\eta(p_\Omega^\mathrm{exact}) -  p_\mathrm{equi}^\eta
    \|_{L^2(\Omega)}=\| I^\eta(p_\Omega^\mathrm{exact}) +I^\eta(p^h_\Omega) -I^\eta(p^h_\Omega)
    -  p_\mathrm{equi}^\eta   \|_{L^2(\Omega)}\\ \leq \| I^\eta(p_\Omega^\mathrm{exact})
    -I^\eta(p^h_\Omega) \|_{L^2(\Omega)}+ \| I^\eta(p^h_\Omega)   -  p_\mathrm{equi}^\eta
    \|_{L^2(\Omega)} \\ \leq \|  I^\eta(p_\Omega^\mathrm{exact}) -I^\eta(p^{h_2}_\Omega)
    \|_{L^2(\Omega)} + \|I^\eta(p^{h_2}_\Omega) -I^\eta(p^h_\Omega) \|_{L^2(\Omega)} +
    \| I^\eta(p^h_\Omega)   -  p_\mathrm{equi}^\eta   \|_{L^2(\Omega)}\\= e_{h_2} + \Delta p + \tilde{e}_m
\end{gather*}
Here $\tilde{e}_m$ is the
previous - incorrect - estimate of the model error, and $\Delta p$ estimates the
effect of grid refinement. If we assume that $e_{h_2}$ is small, since $h_2<h$,
then we state that $e_m\leq \tilde{e}_m + \Delta p$. With similar arguments we
can also obtain $e_m\geq \tilde{e}_m - \Delta p$. The values of this upper and
lower bounds are reported for the three cases in Figure
\ref{fig:err_modello_plot}, where we can observe that the lower bound decreases
linearly with $\epsilon$ as expected.
\begin{figure}
    \centering
    \includegraphics[width=0.32\textwidth]{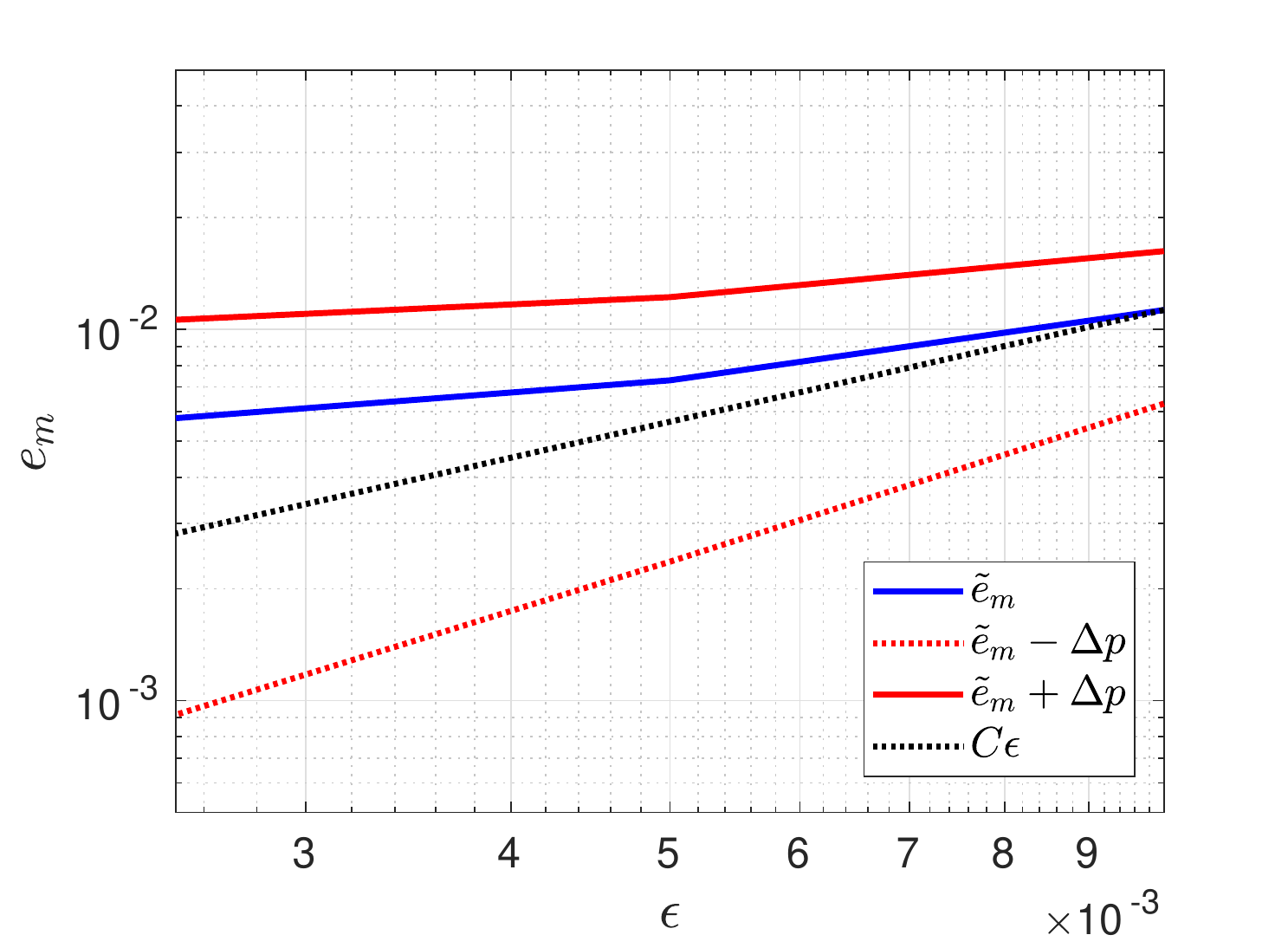}
    \includegraphics[width=0.32\textwidth]{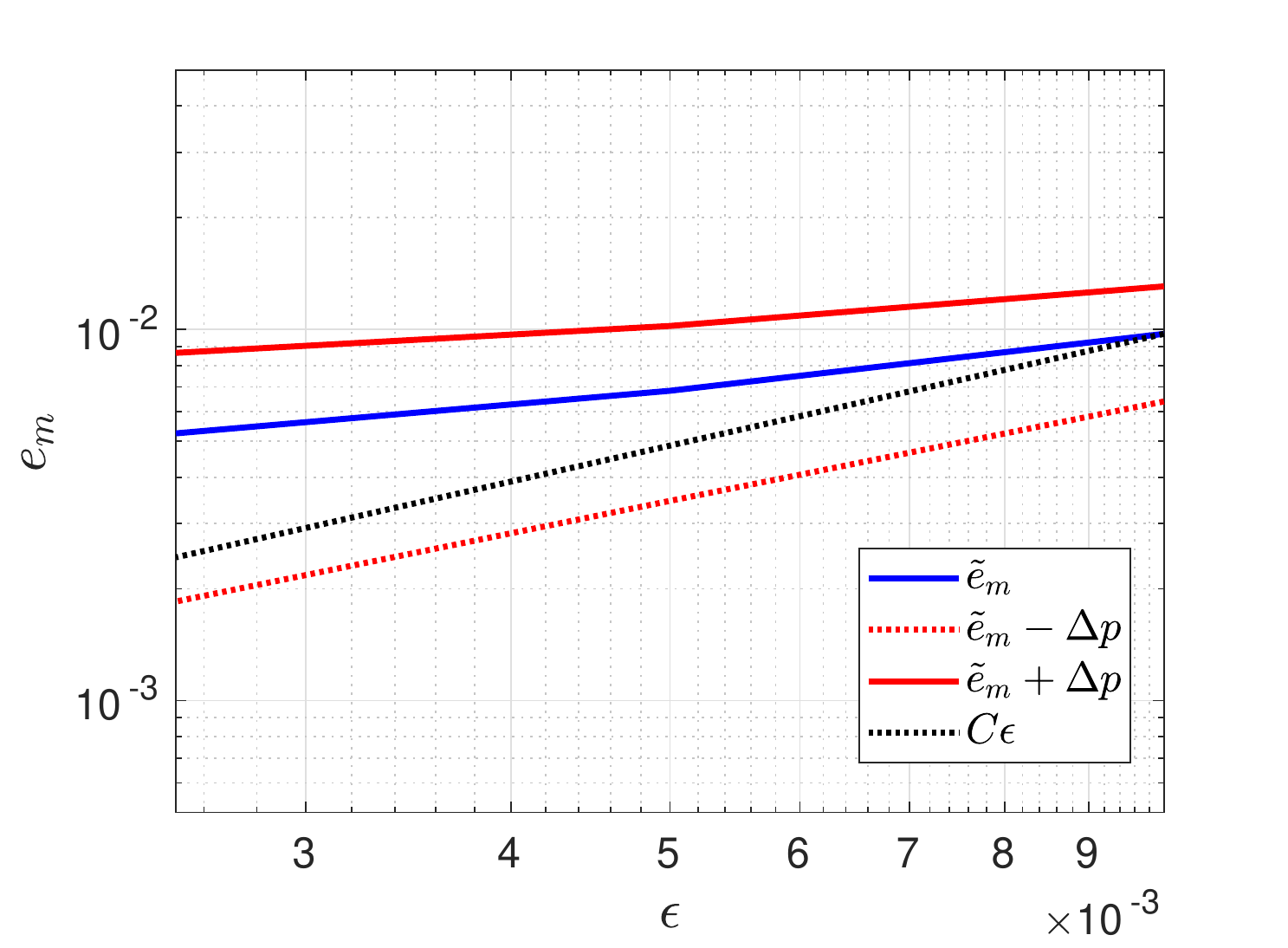}
    \includegraphics[width=0.32\textwidth]{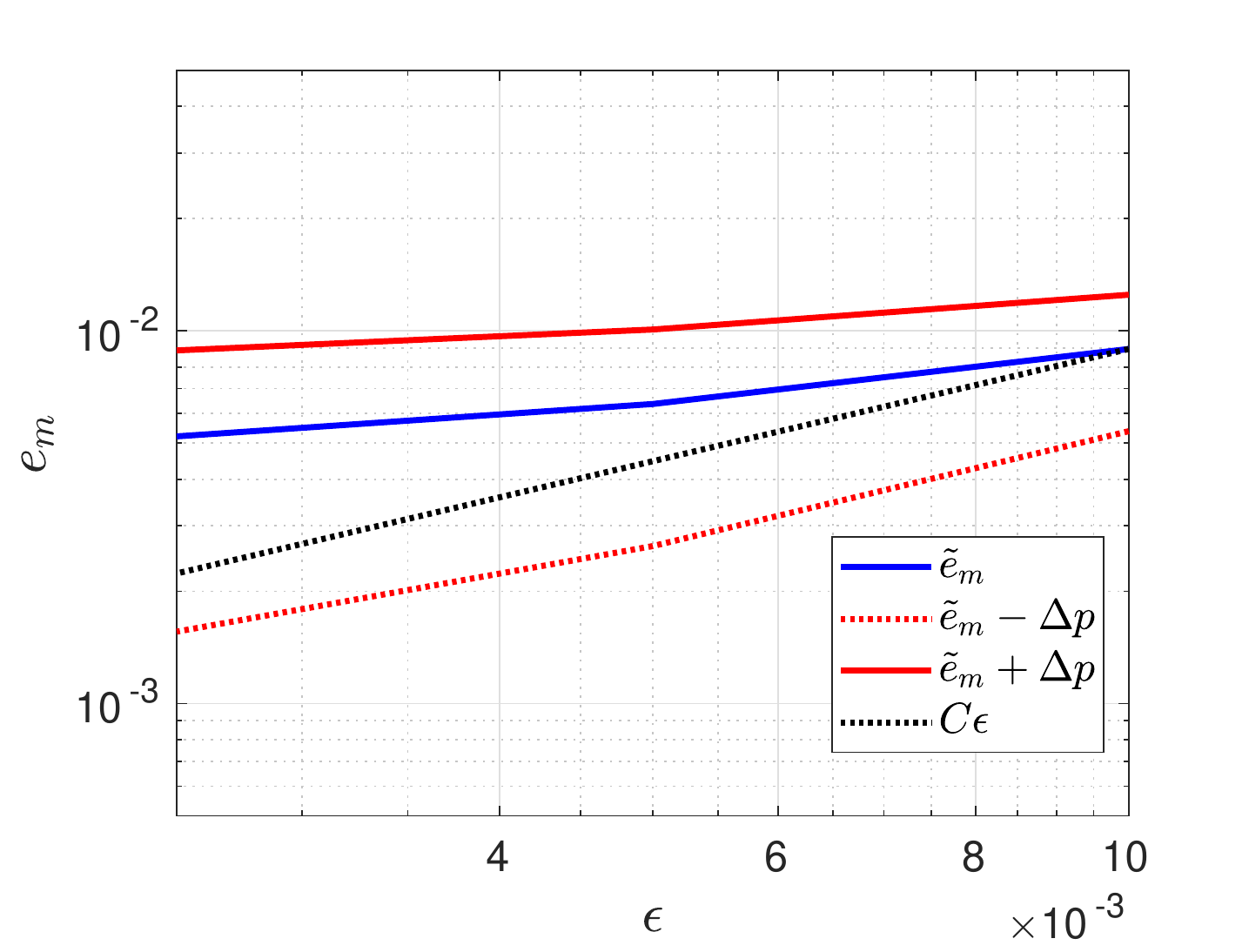}
   \caption{Estimate of the model error, upper and lower bounds compared with a linear
   trend in $\epsilon$ for \emph{case (i)}, \emph{case (ii)} and \emph{case
   (iii)} from left to right.}%
   \label{fig:err_modello_plot}
\end{figure}

\subsection{A three-dimensional example}\label{subsec:example3}

This last test case is inspired by \textit{Case 1: single fracture} of the
benchmark initiative \cite{Berre2018}. We consider a single fault immersed in a
three-dimensional rock matrix, defined as $\Omega = (0, 100)^3$. The fault has
thickness equal to $\epsilon_\gamma = 10^{-3}$ and permeability equal to
$\alpha_\gamma^2 = 10^{-7} \epsilon_\gamma$ and $\alpha_\Gamma^2 = 10^{-7} /
\epsilon_\gamma$. The fault is identified by the following four corners
\begin{gather*}
    c_0 = (0, 0, 80)
    \quad
    c_1 = (100, 0, 20)
    \quad
    c_2 = (100, 100, 20)
    \quad
    c_3 = (0, 100, 80).
\end{gather*}
In the rock matrix the permeability is $\alpha^2_\Omega = 10^{-6}$ for $z \geq 10$,
otherwise $\alpha^2_\Omega = 10^{-5}$. For the boundary condition, we impose 4 for the pressure in
the narrow band $\{0\}\times (0, 100) \times (90, 100)$ and 0 in the portion
$(0, 100) \times \{0\} \times (0, 100)$ of the boundary. The assume zero flux
elsewhere.
For the damage zone we have  a thickness, greater than the fault, equal to $\epsilon_\mu = 10^{-1}$ with
permeability $\alpha_\mu^2 = 10^{-2} \epsilon_\mu$ and $\alpha_M^2 = 10^{-2} /
\epsilon_\mu$ on the upper part and $\alpha_\mu^2 = 10^{-1} \epsilon_\mu$ and
$\alpha_M^2 = 10^{-1} / \epsilon_\mu$ on the lower part.
The numerical solution considers a grid composed of $\sim 9.5k$ tetrahedra for
$\Omega$, $757$ triangles for the fault grid $\gamma$, and $\sim 1.5k$
triangles for the damage zone $\mu$. This correspond to the coarsest mesh of
the benchmark \cite{Berre2018}.
\begin{figure}
    \centering
    \includegraphics[width=0.5\textwidth]{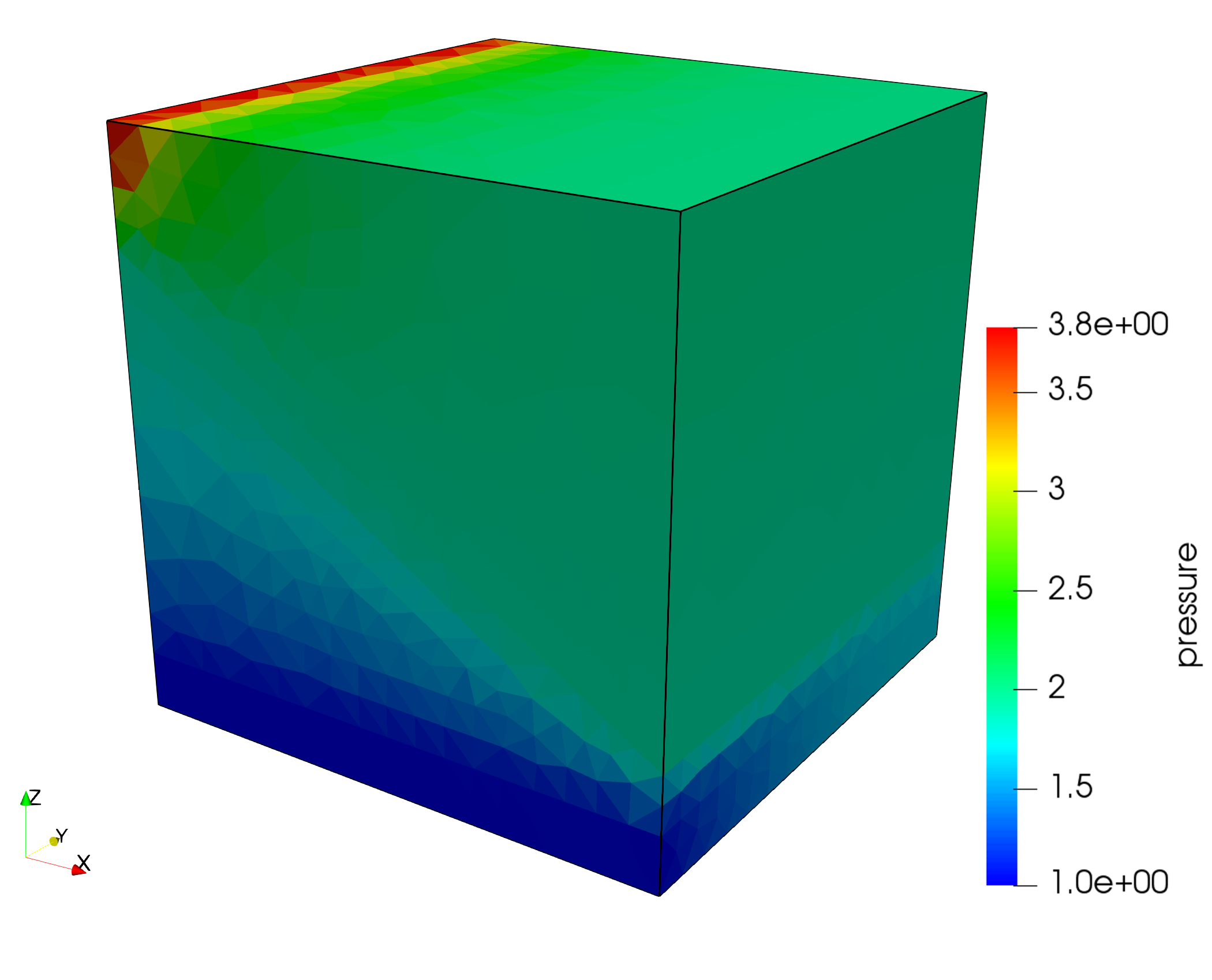}%
    \includegraphics[width=0.5\textwidth]{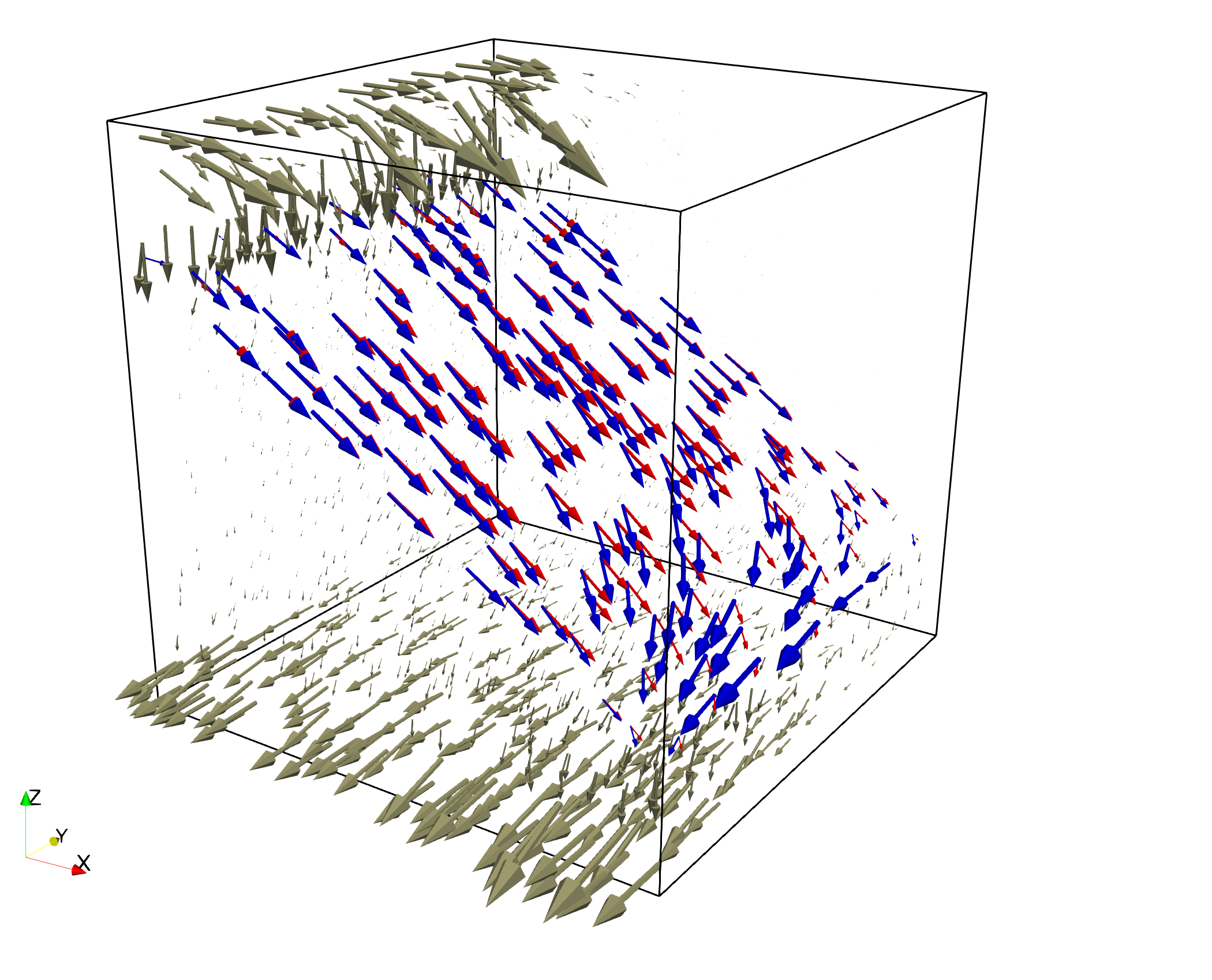}
    \caption{Graphical representation of the solution of the example in
    Subsection \ref{subsec:example3}. On
    the left the pressure field, on the right the velocity field with different
    colours. With grey arrows we indicate the velocity in the rock matrix, with
    blue and red in the upper and lower part of the damage zone, respectively.
    The velocity in the fault is zero, then not represented.}
    \label{fig:example3_solution}
\end{figure}

In Figure \ref{fig:example3_solution}, we present the obtained numerical solution.
On the left the pressure field; since the fault is less permeable than
the damage zone and rock matrix the solution exhibits a steep variation across
the fault. On the right the velocity is represented with different colours for
the rock matrix, fault, and each part (top in blue and bottom in red) of the
damage zone. The
arrows corresponding to the rock matrix and the bottom part of the damage zone
are enlarged by a factor $10^8$, while
for the upper part of the damage zone we use a factor $10^9$.
The velocity along the fault is zero because of its low permeability. Since the
rock matrix is much less permeable than the damage zone the flow tends to be
more concentrated in the latter. The arrows are coherently pointing from the inflow part to
the outflow

Also in this three-dimensional case, the obtained solution is physically sound
and shows the capability and potentiality of the model.

\section{Conclusion}\label{sec:conclusion}

In this work we have introduced a new conceptual model {for
the simulation of Darcy flows in porous media crossed by large faults, i.e. by
complex regions characterized by an inner thin core surrounded by damage zone
where, due to the accommodation of strain, a large number of small fractures is
present. Following a well established literature, we approximate the thin fault
region, and in particular both the core and the damage zone as lower-dimensional
and geometrically coincident objects, i.e. lines in $2d$ and surfaces in $3d$ to
avoid extreme mesh refinement. However, unlike previous works, the presence of
three lower dimensional interfaces instead of a single fault object gives us the
freedom to better characterize the fluid-dynamic behavior of the structure, by
using different permeability and a different thickness for the core and the
damage zone, and even accounting for asymmetries across the fault. Moreover, we
highlight the fact that this approach can be extended, in different areas of
application, to the efficient simulation of thin layered porous media}.

We have proven the well posedness of the weak formulation of the mixed
dimensional problem and discussed its numerical approximation and a set of
examples. The numerical tests confirm the ability of the resulting numerical
scheme to handle high contrasts in permeability among the rock matrix, fault,
and damage zone. Moreover, {by comparing the mixed dimensional solution
with fully resolved equi-dimensional simulations} we have shown that the model
error {associated with geometrical model reduction} is only focused where the
fault or the damage zone give a jump in the pressure field {and, moreover,
this error decreases for thinner faults as expected}. The obtained solutions are
thus physically sound and the model can be regarded as a promising strategy.
{Possible future developments include the possibility to account for
heterogeneities, and changes in time, in the aperture and permeability of the
different layers. These variability could be the result of geochemical or
mechanical processes. Moreover we plan to couple the flow problem with
transport, to fully exploit the freedom and flexibility given by the detailed
description of the permeability of the layers. We remark that the coupling with
transport which is natural thanks to the mixed conservative approximation of the
velocity field.}

\section{Acknowledgements}

We acknowledge the PorePy development team: Runar Berge, Inga Berre, Eirik
Keilegavlen, and  Ivar Stefansson.
Finally, the authors warmly thanks Stefano Scial\`o for many fruitful
discussions.

\bibliographystyle{plain}
\bibliography{biblio}

\end{document}

%% file: domain_equi.pdf_tex
%% Accompanies image file 'domain_equi.pdf' (pdf, eps, ps)
\begingroup%
  \makeatletter%
  \providecommand\color[2][]{%
    \errmessage{(Inkscape) Color is used for the text in Inkscape, but the package 'color.sty' is not loaded}%
    \renewcommand\color[2][]{}%
  }%
  \providecommand\transparent[1]{%
    \errmessage{(Inkscape) Transparency is used (non-zero) for the text in Inkscape, but the package 'transparent.sty' is not loaded}%
    \renewcommand\transparent[1]{}%
  }%
  \providecommand\rotatebox[2]{#2}%
  \newcommand*\fsize{\dimexpr\f@size pt\relax}%
  \newcommand*\lineheight[1]{\fontsize{\fsize}{#1\fsize}\selectfont}%
  \ifx\svgwidth\undefined%
    \setlength{\unitlength}{270.87873324bp}%
    \ifx\svgscale\undefined%
      \relax%
    \else%
      \setlength{\unitlength}{\unitlength * \real{\svgscale}}%
    \fi%
  \else%
    \setlength{\unitlength}{\svgwidth}%
  \fi%
  \global\let\svgwidth\undefined%
  \global\let\svgscale\undefined%
  \makeatother%
  \begin{picture}(1,0.46837621)%
    \lineheight{1}%
    \setlength\tabcolsep{0pt}%
    \put(0,0){\includegraphics[width=\unitlength,page=1]{domain_equi.pdf}}%
    \put(0.32783633,0.44144996){\color[rgb]{0,0,0}\makebox(0,0)[lt]{\lineheight{0}\smash{\begin{tabular}[t]{l}$\mu$\end{tabular}}}}%
    \put(0.21752873,0.00520876){\color[rgb]{0,0,0}\makebox(0,0)[lt]{\lineheight{0}\smash{\begin{tabular}[t]{l}$M$\end{tabular}}}}%
    \put(0,0){\includegraphics[width=\unitlength,page=2]{domain_equi.pdf}}%
    \put(0.09810901,0.2322753){\color[rgb]{0,0,0}\makebox(0,0)[lt]{\lineheight{0}\smash{\begin{tabular}[t]{l}$\Omega$\end{tabular}}}}%
    \put(0,0){\includegraphics[width=\unitlength,page=3]{domain_equi.pdf}}%
    \put(0.26158552,0.25838878){\color[rgb]{0,0,0}\makebox(0,0)[lt]{\lineheight{0}\smash{\begin{tabular}[t]{l}$\bm{n}$\end{tabular}}}}%
    \put(0,0){\includegraphics[width=\unitlength,page=4]{domain_equi.pdf}}%
    \put(0.86911857,0.2322753){\color[rgb]{0,0,0}\makebox(0,0)[lt]{\lineheight{0}\smash{\begin{tabular}[t]{l}$\Omega$\end{tabular}}}}%
    \put(0,0){\includegraphics[width=\unitlength,page=5]{domain_equi.pdf}}%
    \put(0.71120359,0.25838878){\color[rgb]{0,0,0}\makebox(0,0)[lt]{\lineheight{0}\smash{\begin{tabular}[t]{l}$\bm{n}$\end{tabular}}}}%
    \put(0,0){\includegraphics[width=\unitlength,page=6]{domain_equi.pdf}}%
    \put(0.64377817,0.44144996){\color[rgb]{0,0,0}\makebox(0,0)[lt]{\lineheight{0}\smash{\begin{tabular}[t]{l}$\mu$\end{tabular}}}}%
    \put(0.74463519,0.00520876){\color[rgb]{0,0,0}\makebox(0,0)[lt]{\lineheight{0}\smash{\begin{tabular}[t]{l}$M$\end{tabular}}}}%
    \put(0,0){\includegraphics[width=\unitlength,page=7]{domain_equi.pdf}}%
    \put(0.48245136,0.44144996){\color[rgb]{0,0,0}\makebox(0,0)[lt]{\lineheight{0}\smash{\begin{tabular}[t]{l}$\gamma$\end{tabular}}}}%
    \put(0.383588,0.00520876){\color[rgb]{0,0,0}\makebox(0,0)[lt]{\lineheight{0}\smash{\begin{tabular}[t]{l}$\Gamma$\end{tabular}}}}%
    \put(0.59763111,0.00520876){\color[rgb]{0,0,0}\makebox(0,0)[lt]{\lineheight{0}\smash{\begin{tabular}[t]{l}$\Gamma$\end{tabular}}}}%
    \put(0,0){\includegraphics[width=\unitlength,page=8]{domain_equi.pdf}}%
    \put(0.39640373,0.21182676){\color[rgb]{0,0,0}\makebox(0,0)[lt]{\lineheight{0}\smash{\begin{tabular}[t]{l}$\bm{n}_\mu$\end{tabular}}}}%
    \put(0.56144008,0.21183451){\color[rgb]{0,0,0}\makebox(0,0)[lt]{\lineheight{0}\smash{\begin{tabular}[t]{l}$\bm{n}_\mu$\end{tabular}}}}%
  \end{picture}%
\endgroup%

%% file: domain.pdf_tex
%% Accompanies image file 'domain.pdf' (pdf, eps, ps)
\begingroup%
  \makeatletter%
  \providecommand\color[2][]{%
    \errmessage{(Inkscape) Color is used for the text in Inkscape, but the package 'color.sty' is not loaded}%
    \renewcommand\color[2][]{}%
  }%
  \providecommand\transparent[1]{%
    \errmessage{(Inkscape) Transparency is used (non-zero) for the text in Inkscape, but the package 'transparent.sty' is not loaded}%
    \renewcommand\transparent[1]{}%
  }%
  \providecommand\rotatebox[2]{#2}%
  \newcommand*\fsize{\dimexpr\f@size pt\relax}%
  \newcommand*\lineheight[1]{\fontsize{\fsize}{#1\fsize}\selectfont}%
  \ifx\svgwidth\undefined%
    \setlength{\unitlength}{231.8740244bp}%
    \ifx\svgscale\undefined%
      \relax%
    \else%
      \setlength{\unitlength}{\unitlength * \real{\svgscale}}%
    \fi%
  \else%
    \setlength{\unitlength}{\svgwidth}%
  \fi%
  \global\let\svgwidth\undefined%
  \global\let\svgscale\undefined%
  \makeatother%
  \begin{picture}(1,0.54716416)%
    \lineheight{1}%
    \setlength\tabcolsep{0pt}%
    \put(0,0){\includegraphics[width=\unitlength,page=1]{domain.pdf}}%
    \put(0.37004535,0.5157085){\color[rgb]{0,0,0}\makebox(0,0)[lt]{\lineheight{0}\smash{\begin{tabular}[t]{l}$\mu$\end{tabular}}}}%
    \put(0.59681472,0.5157085){\color[rgb]{0,0,0}\makebox(0,0)[lt]{\lineheight{0}\smash{\begin{tabular}[t]{l}$\mu$\end{tabular}}}}%
    \put(0.47950963,0.5157085){\color[rgb]{0,0,0}\makebox(0,0)[lt]{\lineheight{0}\smash{\begin{tabular}[t]{l}$\gamma$\end{tabular}}}}%
    \put(0.26058944,0.00608494){\color[rgb]{0,0,0}\makebox(0,0)[lt]{\lineheight{0}\smash{\begin{tabular}[t]{l}$M$\end{tabular}}}}%
    \put(0.68876128,0.00608494){\color[rgb]{0,0,0}\makebox(0,0)[lt]{\lineheight{0}\smash{\begin{tabular}[t]{l}$M$\end{tabular}}}}%
    \put(0.42870626,0.00608494){\color[rgb]{0,0,0}\makebox(0,0)[lt]{\lineheight{0}\smash{\begin{tabular}[t]{l}$\Gamma$\end{tabular}}}}%
    \put(0.54290502,0.00608494){\color[rgb]{0,0,0}\makebox(0,0)[lt]{\lineheight{0}\smash{\begin{tabular}[t]{l}$\Gamma$\end{tabular}}}}%
    \put(0,0){\includegraphics[width=\unitlength,page=2]{domain.pdf}}%
    \put(0.11461249,0.27134751){\color[rgb]{0,0,0}\makebox(0,0)[lt]{\lineheight{0}\smash{\begin{tabular}[t]{l}$\Omega$\end{tabular}}}}%
    \put(0,0){\includegraphics[width=\unitlength,page=3]{domain.pdf}}%
    \put(0.29911918,0.30185366){\color[rgb]{0,0,0}\makebox(0,0)[lt]{\lineheight{0}\smash{\begin{tabular}[t]{l}$\bm{n}$\end{tabular}}}}%
    \put(0,0){\includegraphics[width=\unitlength,page=4]{domain.pdf}}%
    \put(0.84712271,0.27134751){\color[rgb]{0,0,0}\makebox(0,0)[lt]{\lineheight{0}\smash{\begin{tabular}[t]{l}$\Omega$\end{tabular}}}}%
    \put(0,0){\includegraphics[width=\unitlength,page=5]{domain.pdf}}%
    \put(0.65617502,0.30185366){\color[rgb]{0,0,0}\makebox(0,0)[lt]{\lineheight{0}\smash{\begin{tabular}[t]{l}$\bm{n}$\end{tabular}}}}%
  \end{picture}%
\endgroup%